\documentclass[10pt]{article}
\usepackage{graphicx} 
\usepackage{booktabs}
\usepackage{caption}
\usepackage{subcaption}
\usepackage{amsmath,latexsym,amsfonts,amssymb,mathrsfs,amsthm}

\usepackage{verbatim}

\usepackage{euscript}
\usepackage{algpseudocode,algorithm} 

\usepackage[usenames]{color}
\definecolor{darkred}{RGB}{139,0,0}
\definecolor{darkgreen}{RGB}{0,100,0}
\definecolor{darkmagenta}{RGB}{139,0,139}
\definecolor{darkpurple}{RGB}{110,0,180}
\definecolor{darkblue}{RGB}{40,0,200}
\definecolor{darkorange}{RGB}{255,140,0}

\usepackage[colorlinks=true,linkcolor=black,citecolor=black]{hyperref}
\newcommand{\fa}{{\mathfrak a}}

\def\R{\mathbb{R}}
\def\N{\mathbb{N}}
\def\IC{\mathbb{C}}

\def\IN{\mathbb{N}}

\def\IR{\mathbb{R}}

\def\IE{\mathbb{E}}

\newcommand{\C}{\mathbb{C}}

\newcommand{\Oc}{\mathscr{O}}

\newcommand{\cB}{\mathcal{B}}

\newcommand{\calD}{\mathcal{D}}

\newcommand{\cL}{\mathcal{L}}

\newcommand{\cN}{\mathcal{N}}

\newcommand{\cR}{\mathcal{R}}
\newcommand{\cT}{\mathcal{T}}
\newcommand{\cG}{\mathcal{G}}
\newcommand{\cO}{\mathcal{O}}

\newcommand{\cX}{\mathcal{X}}
\newcommand{\bcX}{\mathcal{X}}
\newcommand{\bcY}{\mathcal{Y}}
\newcommand{\bcZ}{\mathcal{Z}}
\newcommand{\cY}{\mathcal{Y}}

\newcommand{\bg}{{\bf g}}
\newcommand{\bpi}{{\boldsymbol \pi}}
\newcommand{\cZ}{\mathcal{Z}}

\newcommand{\eps}{\varepsilon}

\newcommand{\bsnull}{{\boldsymbol 0}}

\newcommand{\bsbeta}{\boldsymbol{\beta}}

\newcommand{\bsrho}{{\boldsymbol{\rho}}}

\newcommand{\bsb}{{\boldsymbol{b}}}

\newcommand{\bsk}{{\boldsymbol{k}}}

\newcommand{\bsx}{{\boldsymbol{x}}}
\newcommand{\bsy}{{\boldsymbol{y}}}
\newcommand{\bsz}{{\boldsymbol{z}}}

\newcommand{\rd}{\mathrm{d}}
\newcommand{\dd}{\mathrm{d}}
\newcommand{\bbQ}{\mathbb{Q}}
\newcommand{\bbR}{\mathbb{R}}
\newcommand{\bbC}{\mathbb{C}}
\newcommand{\bbJ}{\mathbb{N}}

\newcommand{\bbN}{\mathbb{N}}
\newcommand{\bbE}{\mathbb{E}}

\newcommand{\calO}{\mathcal{O}}
\newcommand{\calM}{\mathcal{M}}

\newcommand{\KL}{Karhunen-Lo\`eve }

\newcommand{\mask}[1]{{}}
\newcommand{\bszero}{{\boldsymbol{0}}}

\numberwithin{equation}{section}

\theoremstyle{plain}
\newtheorem{theorem}{Theorem}[section]
\newtheorem{prop}[theorem]{Proposition}

\newtheorem{assumption}{Assumption}
\theoremstyle{definition}
\newtheorem{definition}{Definition}[section]
\newtheorem{remark}{Remark}[section]

\newcommand{\bz}{{\bf z}}

\newcommand{\bsone}{{\boldsymbol{1}}}

\newcommand{\calS}{\mathcal S}

\newcommand{\be}{\begin{equation}}
\newcommand{\ee}{\end{equation}}
\newcommand{\ba}{\begin{array}}
\newcommand{\ea}{\end{array}}
\newcommand{\beas}{\begin{eqnarray*}}
\newcommand{\eeas}{\end{eqnarray*}}
\newcommand{\bea}{\begin{eqnarray}}
\newcommand{\eea}{\end{eqnarray}}

\title{Multilevel higher order Quasi-Monte Carlo 
      \\
      Bayesian Estimation
}
\author{
Josef Dick, Robert N.~Gantner, Quoc T.~Le Gia, and Christoph Schwab
}

\date{\today}

\begin{document}
\maketitle
\begin{abstract}
We propose and analyze deterministic multilevel approximations for 
Bayesian inversion of operator equations with 
uncertain distributed parameters, subject to additive gaussian 
measurement data. 
The algorithms use a multilevel (ML) approach based on deterministic, 
higher order quasi-Monte Carlo (HoQMC) 
quadrature for approximating the high-dimensional expectations, which arise in the
Bayesian estimators, and a Petrov-Galerkin (PG) method for approximating the solution 
to the underlying partial differential equation (PDE). 
This extends the previous single-level approach from 
[J.~Dick, R.~N.~Gantner, Q.~T.~Le Gia and Ch.~Schwab, 
  Higher order Quasi-Monte Carlo integration for Bayesian Estimation. 
  Report 2016-13, Seminar for Applied Mathematics, ETH Z\"urich (in review)]. 

Compared to the single-level approach, the present convergence analysis of the 
multilevel method requires stronger assumptions on holomorphy and regularity 
of the countably-parametric uncertainty-to-observation maps of the forward problem. 
As in the single-level case and in the affine-parametric case analyzed in 
[ J.~Dick, F.Y.~Kuo, Q.~T.~Le Gia and Ch.~Schwab,
  Multi-level higher order QMC Galerkin discretization for affine parametric operator equations. 
Accepted for publication in SIAM J. Numer. Anal., 2016], 
we obtain sufficient conditions which allow us to achieve arbitrarily high, algebraic
convergence rates in terms of work, 
which are independent of the dimension of the parameter space. 
The convergence rates are limited only by the spatial regularity of the forward problem,
the discretization order achieved by the Petrov Galerkin discretization, and 
by the sparsity of the uncertainty parametrization.

We provide detailed numerical experiments for linear elliptic problems in two space dimensions, with $s=1024$ parameters characterizing the uncertain input, 
confirming the theory and showing that the ML HoQMC algorithms outperform, 
in terms of error vs.~computational work, 
both multilevel Monte Carlo (MLMC) methods and single-level (SL) HoQMC methods.
\end{abstract}

\newpage
\setcounter{page}{0}
\tableofcontents
\newpage
\section{Introduction}
\label{sec:Intro}
In \cite{DGLGCSBIPSL} we proposed and analyzed the convergence rates of higher order Quasi-Monte Carlo (HoQMC) approximations of conditional expectations which arise in Bayesian Inverse problems for partial differential equations (PDEs). 
We studied broad classes of parametric operator equations
with \emph{distributed uncertain parametric input data}. 
Typical examples are elliptic or parabolic partial differential equations with uncertain, 
spatially inhomogeneous coefficients, but also differential and integral equations in uncertain domains of definition. 
Upon suitable \emph{uncertainty parametrization}, and with a suitable Bayesian prior measure placed on the, 
in general, infinite-dimensional parameter space, the task of numerical evaluation 
of Bayesian estimates for quantities of interest (QoI's) becomes that of 
numerical computation of parametric, deterministic integrals over a high-dimensional parameter space. 
As an alternative to the Markov chain Monte Carlo (MCMC) method, in \cite{SS12,SS13} 
it was proposed to apply recently developed, dimension-adaptive 
Smolyak quadrature techniques to the evaluation of the corresponding integrals. 
In \cite{DGLGCSBIPSL} we developed a convergence theory for HoQMC 
integration for the numerical evaluation of the corresponding integrals, 
based on our earlier work \cite{DKGNS13} on these methods in forward uncertainty quantification (UQ). 
In particular, we proved that \emph{dimension-independent} convergence rates of order $> 1/2$ in terms of the number $N$ of approximate solves 
of the forward problem can be achieved by replacing Monte Carlo or MCMC sampling 
of the Bayesian posterior with 
judiciously chosen, \emph{deterministic} HoQMC quadratures. 
The achievable, dimension-independent rate in the proposed algorithms is, in principle, 
as high as the sparsity of the forward map permits, and ``embarrassingly parallel'': 
being QMC algorithms they access, unlike MCMC and sequential Monte Carlo (SMC) methods, 
the forward problem simultaneously and in parallel.
The error analysis in \cite{DGLGCSBIPSL} accounted for 
the quadrature error as well as for the errors incurred by 
a Petrov-Galerkin (PG) discretization and dimension truncation
of the forward problem, but
was performed for a single-level algorithm, 
i.e.~the PG discretization of the forward model in
all QMC quadrature points was based on the same subspace.
As is well known in the context of Monte Carlo methods,
\emph{multilevel strategies} can lead to substantial gains
in accuracy versus work. We refer to \cite{GilesActa}, and the 
references there, for a survey on multilevel Monte-Carlo methods.
Multilevel discretizations for QMC integration were explored first for 
parametric, linear forward problems in \cite{KSS13,HPS16,KSSSU15} and, 
in the context of HoQMC for parametric operator equations, in \cite{DKGS14}.
For the use of multilevel strategies 
in the context of MCMC methods for Bayesian inverse problems
we refer to \cite{DodSchlTeck,HoangScSt12} and the references there.
The purpose of the present paper is to extend the convergence analysis 
of the HoQMC, Petrov-Galerkin approach from \cite{DGLGCSBIPSL} to a multilevel algorithm. 

As in our single-level HoQMC PG error analysis of Bayesian inversion
in \cite{DGLGCSBIPSL}, we adopt the abstract setting 
of Bayesian inverse problems in infinite-dimensional function spaces
from \cite{Stuart13}, 
and the convergence analysis of PG 
discretizations of abstract,  nonlinear parametric problems from
\cite{BreRapRav80,PR}, as reviewed in \cite{DGLGCSBIPSL}; 
\emph{throughout the present work, we adopt the notations
and terminology from} \cite{DGLGCSBIPSL}.

The principal contributions of the present work are as follows:
we derive, for the possibly nonlinear, parametric operator equations with
distributed, uncertain input considered in \cite{DGLGCSBIPSL},
multilevel extensions of high-order Petrov-Galerkin discretizations
of the countably-parametric forward problem combined with HoQMC
integration in the parameter domain.
We provide a complete convergence analysis of the proposed algorithm,
specifying, in particular, precise regularity and sparsity conditions
on the forward response map which are sufficient to achieve a 
certain, dimension-independent convergence rate. 
Our analysis provides,
in particular, information on the choice of algorithm parameters
in applications: the convergence order of the PG discretization,
also for functionals of the solution (not considered in \cite{DGLGCSBIPSL}),
the order of the interlaced polynomial lattice rule, 
the relation of HoQMC sample numbers $N_\ell$ and 
of truncation dimensions $s_\ell$ of the parameter space on the 
discretization level $\ell$ of the PG approximation of the forward
problem.
A major conclusion obtained with analytic continuation based 
error analysis from \cite{DLGCS14} is that
\emph{identical algorithmic steering parameters are admissible in our
      ML HoQMC algorithm for both, forward and inverse UQ}.
With optimized HoQMC sample numbers, we 
also obtain asymptotic error vs.~work bounds
which indicate that the presently proposed algorithms
outperform both, multilevel MC as well as the 
multilevel first order QMC algorithm considered in \cite{KSS13};
our analysis reveals that, analogous to sparse grid approximation,
stronger regularity requirements on the parametric forward problem 
both in physical space and in parameter space are required.

In a suite of numerical experiments, we provide PDE examples
in two space dimensions and with parameter spaces of dimension of several
thousand where the presently proposed ML algorithms
outperform all mentioned methods in terms of error vs.~work,
starting at relative errors as large as $10\%$.

The structure of this paper is as follows. 
In Section \ref{sec:HolOpEq}, we present abstract,
nonlinear parametric operator equations with uncertain input data, 
the parametrization of the uncertain input data and PG approximations
of their parametric solutions.
The setting is analogous to that from \cite{DGLGCSBIPSL}
and the presentation is thus synoptic, and analogous 
and with references to \cite[Sections 2, 3]{DGLGCSBIPSL}.
In Section \ref{sec:BayInv}, we review the general theory of
well-posed Bayesian inverse problems in function spaces,
from \cite{Stuart13}. Again, the material is analogous to
what we used in \cite{DGLGCSBIPSL}; however, the error analysis of the ML HoQMC
method, being a form of sparse grid approximation, 
requires the analog of mixed regularity, which we develop. Section \ref{sec:MLHOQMC} contains the core new mathematical
results of the present paper: several multilevel
computable Bayesian estimators and their error analysis.
In particular it contains the error vs.~work analysis of the combined 
ML HoQMC Petrov-Galerkin algorithms.
Section \ref{sec:NumExp} presents specific examples of parametric
forward problems, and verifies that they satisfy all hypotheses
of our foregoing error analysis; 
specifically, we consider linear, affine-parametric
diffusion problems in two space dimensions, 
in primal variational formulation
with space of continuous, piecewise linear functions
on regular triangulations.
Section \ref{sec:results} contains numerical tests 
of the proposed estimators 
for forward and inverse UQ for the PDE problems 
considered in Section \ref{sec:NumExp}. 
The numerical results are in agreement with the theory.
\section{Forward UQ for parametric operator equations}
\label{sec:HolOpEq}
We review the notation and mathematical setting of
forward and inverse UQ for a class of 
smooth, possibly nonlinear, parametric operator equations, for which we 
developed the error analysis of the single-level algorithm in \cite{DGLGCSBIPSL}.
We develop here the error analysis
for the multilevel extension of the algorithms in 
\cite{DGLGCSBIPSL} for a general class of 
forward problems given by smooth, 
nonlinear operator equations
with input data $u$ from a separable Banach space $X$. 
Upon uncertainty parametrization with an unconditional
basis of $X$ such as, for example, the \KL basis,
both forward and (Bayesian) inverse problems
become countably parametric, deterministic operator equations.
The problems of forward and inverse UQ are 
reformulated as countably-parametric integration problems. In the present work, we focus on the latter
and analyze the use of \emph{deterministic, higher order Quasi Monte-Carlo} 
integration methods, from \cite{DKGNS13,DKGS14,DLGCS14} 
and the references there, in multilevel algorithms for Bayesian estimation in
partial differential equations with uncertain input.

\subsection{Uncertainty parametrization}
\label{sec:Param}
As in \cite{DGLGCSBIPSL}, 
we parametrize the distributed uncertain input data $u$. 
To this end, $X$ is assumed to be a separable, infinite-dimensional Banach space 
with norm $\|\cdot\|_X$, which has an unconditional basis $\{\psi_j\}_{j\geq 1}$: 
$X = {\rm span}\{\psi_j: j\geq 1\}$. Let $\langle u \rangle \in X$ be fixed. 
Any $u \in X$ has the representation
\begin{equation}\label{eq:uviapsi}
u = u(\bsy) \sim \langle u \rangle + \sum_{j\geq 1} y_j \psi_j 
\quad \mbox{for some } \bsy = (y_j)_{j \geq 1} \in \mathbb{R}^{\mathbb{N}}
\;,
\end{equation}
where $\sim$ means convergence in the norm $\|\cdot\|_X$ of $X$.
The \KL expansions (see, e.g., \cite{ST06,SchwabGittelsonActNum11,Stuart10,Stuart13}) 
are examples of representations \eqref{eq:uviapsi}.
We remark that the representation \eqref{eq:uviapsi} is not unique: 
rescaling $y_j$ and $\psi_j$ will not change $u$.

Assume that a smoothness scale $\{ X_r \}_{r \geq 0}$,
with $X=X_0 \supset X_1 \supset X_2 \supset ...$, and a $t \ge 0$ is being given
as part of the problem specification.
We restrict the uncertain inputs
$u$ to sets $X_t$  with ``higher regularity'' 
in order to obtain convergence rate estimates for the discretization
of the forward problem. 
Note that $u\in X_t$ often corresponds to stronger decay of 
the $\psi_j$ in \eqref{eq:uviapsi}.
We assume that the $\{\psi_j\}_{j\geq 1}$ are scaled such that 
\begin{equation}\label{eq:psumpsi}
\sum_{j\geq 1} \| \psi_j \|_{X_r}^{p_r} <\infty\;, \quad \mbox{for some } 0 < p_r < 1 \mbox{ and } r = 0, \ldots, t  \;.
\end{equation}
We define the set $\tilde{X}$ by
\begin{equation}\label{eq:DefSetU}
\tilde{X} = \{u(\bsy) \in X:  
u(\bsy) \sim \langle u \rangle + \sum_{j \geq 1} y_j \psi_j: \bsy = (y_j)_{j \ge 1} \in U \},
\end{equation}
where $U = (-1, 1)^{\mathbb{N}}$. 
Further we define the sequences $\bsb_r = (b_{j,r})_{j\geq 1}$  
by $b_{j,r} :=  \| \psi_j \|_{X_r}$ for $r = 0, \ldots, t$. 
If $t=0$, we write $X_0=X$ and $\bsb_0 = \bsb$.
From \eqref{eq:psumpsi} we have $\sum_{j \ge 1} b_{j,r}^{p_r} < \infty$ 
for $r = 0, \ldots, t$. 
We also assume that the $\psi_j$ are enumerated so that 
\begin{equation} \label{eq:ordered} 
b_{1,t} \ge b_{2,t} \ge \cdots \ge b_{j,t} \ge \, \cdots\;.
\end{equation}

With a given unconditional basis $\{\psi_j\}_{j\geq 1}$,
realizations of $u\in X$ correspond to pairs $(\langle u \rangle,\bsy)$,
where $\langle u \rangle$ 
is the nominal value of uncertain data $u$
and $\bsy$ is the coordinate vector that determines 
the unique representation \eqref{eq:uviapsi}.
\subsection{Operator equations with uncertain input}
\label{sec:OpEqUncInp}
We consider the abstract forward problem: 
\be\label{eq:NonOpEqn}
\mbox{given}\;\;   u \in \tilde{X}: \,\; 
\mbox{find} \; q\in \bcX \quad \mbox{s.t.} \quad 
_{\bcY'}\langle \cR(u;q) , v \rangle_\bcY = 0  \quad \forall v\in \bcY.
\ee
Here, $X, \bcX$ and $\bcY$ are real, separable Banach spaces and 
$\cR: X \times \bcX \rightarrow \bcY'$ is the residual of a \emph{forward operator}.

A solution $q_0$ of \eqref{eq:NonOpEqn} is called {\em regular at $u$},
for a given $u \in \tilde{X}$, 
if and only if the map $\cR(u;\cdot)$ is 
Fr\'{e}chet differentiable with respect to $q$ at $q_0$
and if the differential is an isomorphism between $\bcX$ and $\bcY'$. We assume the map $\cR(u;\cdot):\bcX\to \bcY'$ 
admits a family of regular solutions 
{\em locally, in an open neighborhood of the nominal parameter instance $\langle u \rangle \in X$}, 
so that the operator equations involving $\cR(u;q)$ are well-posed. 
A particular structural assumption on $\cR$ is the representation
\begin{equation}\label{eq:main}
\cR(u;q) = A(u;q) - F(u) \quad \mbox{in}\quad \bcY'
\end{equation}
with Frechet differentiable maps 
$A: X\times \bcX \to \bcY'$ and $F: X \to \bcY'$.
The set $\{(u,q(u)): u\in \tilde{X} \} \subset X\times \bcX$ 
is called {\em regular branch of solutions} of \eqref{eq:NonOpEqn} 
(in the sense of \cite{BreRapRav80}) if
\begin{equation}\label{eq:RegBranch}
\begin{array}{l}
u \mapsto q(u) \;\mbox{{is a continuous function from}} \; X \mbox{into} \; \bcX\;,
\\
\cR(u;q(u)) = 0 \quad \mbox{in}\quad \bcY'
\;.
\end{array}
\end{equation}
The regular branch of solutions \eqref{eq:RegBranch} 
is called {\em nonsingular} if, in addition, the differential
\begin{equation}\label{eq:NonSingBranch}
(D_q\cR)(u;q(u))\in \cL(\bcX,\bcY') \;
\mbox{is an isomorphism from $\bcX$ onto $\bcY'$, for all $u\in \tilde{X}$}
\;.
\end{equation}
Conditions for well-posedness of \eqref{eq:NonOpEqn} 
are stated in 
\cite[Proposition 2.1]{DLGCS14}:
for regular branches of nonsingular solutions 
given by \eqref{eq:NonOpEqn} - \eqref{eq:NonSingBranch},
the nonsingularity of the differential
$(D_q\cR)(u,q(u))$ implies the {\em inf-sup conditions}.
%
%
Under the inf-sup conditions,
for every instance $u\in \tilde{X}$, there exists a unique, isolated solution $q(u)$ of 
\eqref{eq:NonOpEqn}, which is uniformly bounded 
in the sense that there exists
a constant $C(\cR,\tilde{X}) > 0$ such that 
\begin{equation}\label{eq:LocBdd}
\sup_{u\in \tilde{X}} \| q(u) \|_{\bcX} \leq C(\cR,\tilde{X})
\;.
\end{equation}
The set 
$\{ (u,q(u)): u\in \tilde{X}\}\subset \tilde{X}\times \bcX$
is called a {\em regular branch of nonsingular solutions}.

At every point of the regular branch
$\{ (u,q(u)): u\in \tilde{X}\}\subset \tilde{X}\times \bcX$,
if the nonlinear functional $\cR$ is 
Fr\'{e}chet differentiable with respect to $u\in \tilde{X}$ and 
Fr\'{e}chet differentiable with respect to $q$,
then the mapping relating $u$ to $q(u)$ within the branch
of nonsingular solutions is locally Lipschitz on ${\tilde{X}}$: 
\begin{equation}\label{eq:LocLip}
\forall u,v\in \tilde{X}:
\quad 
\| q(u) - q(v) \|_{\bcX} \leq L(\cR,\tilde{X}) \| u-v \|_X 
\;.
\end{equation}

In what follows, we consider the abstract
setting \eqref{eq:NonOpEqn} with 
the assumption that the mapping $\cR(u;q)$ is
uniformly continuously differentiable
with boundedly invertible differential
in a product of neighborhoods
$B_X(\langle u \rangle;R)\times B_\bcX(q(\langle u \rangle);R)
 \subset X\times \cX$, where $B$ denotes the ball with radius $R$, of sufficiently small radius $R>0$.
Then $q_0 = q(\langle u \rangle) \in \bcX$ is
the corresponding unique, isolated solution of \eqref{eq:NonOpEqn} 
at the nominal uncertain input $\langle u \rangle\in X$.
This turns \eqref{eq:NonOpEqn} into an 
{\em equivalent, deterministic, countably parametric operator equation}:
given $\bsy\in U$, find $q(\bsy)\in \cX$ such that
\begin{equation}\label{eq:paraOpEq}
\cR(\bsy;q(\bsy)) = 0 \quad\mbox{in}\quad \cY' \;.
\end{equation}
We refer to \cite[Remark. 2.1]{DGLGCSBIPSL} for further discussion.
Under \eqref{eq:LocBdd} and \eqref{eq:LocLip}, 
the operator equation \eqref{eq:NonOpEqn} will admit
a unique solution $q(\bsy;F)$ for every $\bsy\in U$. 
This solution is,
due to \eqref{eq:LocBdd}, uniformly bounded,
\begin{equation}\label{eq:ParDepy}
\sup_{\bsy\in U} \| q(\bsy) \|_{\cX} \leq C(\cR,U),
\end{equation}
and due to \eqref{eq:LocLip}, 
it depends Lipschitz continuously
on the parameter sequence $\bsy\in U$: 
%
%
there exists a Lipschitz constant $L>0$ such that 
\begin{equation}\label{eq:ParLip}
\forall \bsy,\bsy'\in U:\quad 
\| q(\bsy) - q(\bsy') \|_\cX \leq L(\cR,U) \| u(\bsy) - u(\bsy') \|_X
\;.
\end{equation}
The Lipschitz constant $L>0$ in \eqref{eq:ParLip} 
is generally different from the constant
$L(\cR,\tilde{X})$ in \eqref{eq:LocLip}, as it depends
on $\langle u\rangle\in X$ and on the basis $\{\psi_j\}_{j\geq 1}$.
\emph{
Unless explicitly stated otherwise, 
throughout what follows, we shall identify 
$q_0 = q(\bsnull) \in \cX$ with the solution of \eqref{eq:NonOpEqn} at the
nominal input $\langle u \rangle \in X$.
}

Note that $u\in X_t$ implies that $\cR(\cdot)\in \bcY'_t$ and $q(u)\in \bcX_t$, 
with corresponding subspaces $\bcX_t\subset \bcX$ and $\bcY'_t\subset \bcY'$
with extra regularity from suitable scales. 
If $t=0$, we write  $\bcX_0 = \bcX$, $\bcY_0 = \bcY$, etc.. 
\subsection{Dimension truncation}
\label{sec:dimtrunc}
Dimension truncation is equivalent to setting $y_j=0$ 
for $j>s$, for a given $s\in \N$, in \eqref{eq:uviapsi}.
The truncated uncertain datum is denoted by $u^s \in X$.
We denote by $q^s(\bsy)$ the solution of the corresponding
parametric weak problem \eqref{eq:paraOpEq}. 
For $\bsy\in U$, define 
$\bsy_{\{1:s\}} := (y_1,y_2,...,y_s,0,0,...)$. 
Then unique solvability of \eqref{eq:paraOpEq} implies 
$q^s(\bsy) = q(\bsy_{\{1:s\}})$.
Consider the $s$-term truncated parametric problem: 
given $u^s = u(\bsy_{\{ 1: s \}}) \in \tilde{X}$,
\begin{equation}\label{eq:mainstrunc}
\mbox{find}\;q^s\in \bcX:\quad 
{ _{\bcY'} \langle \cR(u^s;q^s), w \rangle_{\bcY} } = 0 
\;\;\forall w \in \bcY 
\;.
\end{equation}
As shown in \cite[Prop.~2.2]{DGLGCSBIPSL}, \cite[Thm.~5.1]{KSS12},
under Assumption~\eqref{eq:psumpsi},
and under the uniform regularity shift in \eqref{eq:UPR} ahead,
for every $F \in \cY_t'$, 
for every $s\in\bbN$ and for every $\bsy\in U$
the parametric solution $q^s(\bsy)$ of the 
truncated parametric weak problem \eqref{eq:paraOpEq} 
with $u$ replaced by $u^s$ satisfies, 
with $b_{t,j}$ as defined in \eqref{eq:psumpsi},
\begin{equation}\label{eq:Vdimtrunc}
\sup_{\bsy\in U}
\| q(\bsy) - q^s(\bsy) \|_{\bcX_t} \,\le\, C(\cR,X) \sum_{j\ge s+1} b_{t,j}
\;.
\end{equation}
Moreover, for every $G \in \cX_t'$,
there exists $\theta \in \{ 1,2 \}$ such that
\begin{equation}\label{eq:Idimtrunc}
  |I(G(q))- I(G(q^s))|
  \,\le\, \tilde{C} 
  \left( \sum_{j\ge s+1} b_{t,j} \right)^\theta
\end{equation}
where
\begin{equation*}
I(G(q)) = \int_U G(q(\bsy)) \, \bpi(\mathrm{d} \bsy)
\mbox{ and } 
I(G(q^s)) = \int_{[-1,1]^s} G(q(y_1,\ldots, y_s, 0,\ldots)) \,
\bpi(\mathrm{d}y_1 \cdots \,\mathrm{d} y_s),
\end{equation*}
for some constant $\tilde{C}>0$ independent of $s$.
\begin{remark}\label{rmk:dimtrc}
In \eqref{eq:Idimtrunc}, generally $\theta = 1$.
There holds $\theta = 2$ if
$\int_{-1}^1 y_j \pi_j(\mathrm{d}y_j) = 0$, as, e.g.
for the uniform measure (see, e.g.~\cite[Thm.~5.1]{KSS12}).
If conditions \eqref{eq:psumpsi} 
and \eqref{eq:ordered} hold, 
then in \eqref{eq:Vdimtrunc} and \eqref{eq:Idimtrunc} 
\begin{equation}\label{eq:DTbound}
  \sum_{j\ge s+1} b_{t,j}
  \,\le\,
  \min\left(\frac{1}{1/p_t - 1},1\right)
  \bigg(\sum_{j\ge1} b_{t,j}^{p_t} \bigg)^{1/p_t} s^{-(1/p_t - 1)}
  \;.
\end{equation}
\end{remark}
%
\subsection{Petrov-Galerkin discretization}
\label{sec:Discr}
In \cite[Chap. IV.3]{GR90} and in \cite{PR}, a convergence rate analysis of 
Petrov-Galerkin discretizations of regular branches of solutions
for smooth, nonlinear forward problems \eqref{eq:NonOpEqn}
was developed. This setting was adopted in \cite{DGLGCSBIPSL} and
is also the basis for the presently considered multi-level extension
of the single-level version of the present work \cite{DGLGCSBIPSL}.
As in \cite{ScMCQMC12,DKGNS13}, we assume that we are given two sequences
$\{ \cX^h \}_{h>0}\subset\cX$ and $\{ \cY^h \}_{h>0}\subset\cY$ of 
finite dimensional subspaces which are dense in $\cX$ and in $\cY$, respectively. 
We assume the following 
{\em regularity properties} and {\em approximation properties}: 

\noindent
{\bf uniform parametric regularity property (UPR)}: 
there are scales $\{ \bcX_t \}_{t\geq 0}$ and $\{ \bcY_t \}_{t\geq 0}$
of function spaces such that 
$\cX_{t'} \subset \cX_t \subset \cX_0 = \cX$ 
and 
$\cX'_{t'} \subset \cX'_t \subset \cX_0' = \cX'$
for any $0<t<t'<\infty$ 
and analogously for $\bcY_t$, such that there holds the uniform regularity
shift:
for $\cR$ as in \eqref{eq:main},
for every $F\in \bcY'_{t}$, $G\in \bcX'_{t}$, 
the parametric solutions
$q(\bsy) = (A(\bsy))^{-1}F$ and $q^*(\bsy) = (A^*(\bsy))^{-1} G$,
where $A^*$ is the conjugate operator of $A$, satisfy
regularity resp. adjoint regularity shifts which are uniform w.r. to $\bsy$, 
i.e.
\begin{equation}\label{eq:UPR}
\sup_{\bsy \in U} \| q(\bsy) \|_{\bcX_t} \leq C(t) \| F \|_{\bcY'_t} \;,
\quad 
\sup_{\bsy \in U} \| q^*(\bsy) \|_{\bcY_t} \leq C(t) \| G \|_{\bcX'_t} 
\;.
\end{equation}
We also assume {\bf approximation properties}:
for $0 < t,t' \leq \bar{t}$ and for $0< h \leq h_0$ holds
\begin{equation} \label{eq:apprprop}
\inf_{w^h\in \cX^h} \| w - w^h \|_{\cX}
\,\leq\, C_t\, h^t\, \| w \|_{\cX_t} \;,
\;\;
\inf_{v^h\in \cY^h} \| v - v^h \|_{\cY}
\,\leq\, C_{t'}\, h^{t'}\, \| v \|_{\cY_{t'}} 
\;.
\end{equation}
%
Assume that the subspace range
$\{ \cX^h \}_{h>0}\subset\cX$ and $\{ \cY^h \}_{h>0}\subset\cY$ are stable
in the following sense:
there is $\bar{\mu} > 0$ and a discretization parameter $h_0 > 0$
such that for every $0<h \leq h_0$, 
the discrete inf-sup conditions hold uniformly (with respect to $\bsy\in U$)
\begin{align}\label{eq:Bhinfsup1}
&\forall \bsy \in U:
\quad
\inf_{0\ne v^h \in \bcX^h} \sup_{0\ne w^h \in \bcY^h}
\frac{
 _{\bcY'}\langle  (D_q\cR)(u(\bsy);q_0)v^h,w^h\rangle_\bcY
}{
\| v^h \|_\bcX \| w^h \|_{\bcY}
}
\geq \bar{\mu} > 0\;,
\\
\label{eq:Bhinfsup2}
&\forall \bsy\in U:\quad
\inf_{0\ne w^h \in \bcY^h} \sup_{0\ne v^h \in \bcX^h}
\frac{_{\bcY'}\langle (D_q\cR)(u(\bsy);q_0)v^h,w^h\rangle_\bcY}
     {\| v^h \|_{\bcX} \|w^h\|_{\bcY}}
\geq \bar{\mu}>0
\;.
\end{align}
Then, for every $0<h \leq h_0$, 
the solutions $q_h$ of the Petrov-Galerkin approximation problem:
\begin{equation} \label{eq:parmOpEqh}
\mbox{ given }\bsy\in U, \;
\mbox{find} \; q_h(\bsy) \in \bcX^h :
\quad
{_{\bcY'}}\langle \cR(\bsy; q_h(\bsy)), w^h \rangle_{\bcY} = 0
\quad 
\forall w^h\in \bcY^h\;,
\end{equation}
are uniquely defined and converge quasioptimally:
there is a positive constant $C$ such that for all $\bsy\in U$
\begin{equation} \label{eq:quasiopt}
 \| q(\bsy) - q_h(\bsy) \|_{\cX}
 \,\le\, 
\frac{C}{\bar{\mu}} \inf_{0\ne v^h\in \cX^h} \| q(\bsy) - v^h\|_{\cX}
\;.
\end{equation}
If $q(\bsy)\in \bcX_t$ for all $\bsy \in U$
and if \eqref{eq:apprprop} holds, then for every $\bsy\in U$ 
\begin{equation}\label{eq:convrate}
\| q(\bsy) - q_h(\bsy) \|_{\cX}
\,\le\, 
\frac{C}{\bar{\mu}} h^t \sup_{\bsy\in U} \| q(\bsy) \|_{\bcX_t} 
\;.
\end{equation}
Moreover, 
for sufficiently large truncation dimension
$s\in \IN$, for given $\bsy_{\{1:s\}}\in U$ 
the solutions of the following truncated Galerkin problems
\begin{equation} \label{eq:parmOpEqh_trun}
\mbox{find} \; q_h(\bsy_{\{1:s\}}) \in \bcX^h :
\quad
{{_{\bcY'}}\langle  \cR(\bsy_{\{1:s\}}; q_h(\bsy_{\{1:s\}})), w^h \rangle_{\bcY} } = 0
\quad 
\forall w^h\in \bcY^h
\;,
\end{equation}
admit unique solutions $q_h(\bsy_{\{1:s\}}) \in \bcX^h$ which
converge quasioptimally to $q(\bsy_{\{1:s\}}) \in \bcX$ as $h\downarrow 0$,  
i.e.~\eqref{eq:quasiopt} and \eqref{eq:convrate} hold 
with $\bsy_{\{1:s\}}$ in place of $\bsy$, with the same constants
$C>0$ and $\bar{\mu}$ independent of $s$ and of $h$.
%
\section{Bayesian inverse UQ and HoQMC approximation}
\label{sec:BayInv}
%
In this section, the abstract (possibly nonlinear) operator equation \eqref{eq:NonOpEqn}
is considered again.
The system's forcing $F\in \cY'$ 
is allowed to depend on the uncertain input $u$ and
the uncertain operator $A(u;\cdot)\in \cL(\cX,\cY')$ 
is boundedly invertible, for the uncertain input $u$,
sufficiently close to a nominal instance of the
uncertain input data $\langle u\rangle \in X$. 
We define the {\em forward response map},
which maps a given uncertain input $u$ and a given forcing $F$
in \eqref{eq:main} to the response $q$ in \eqref{eq:NonOpEqn} by
\begin{align*}
 G: X\times \bcY' &\to \cX \;: \quad G(u,F)        := q(u)\;.
\end{align*}
In the general case \eqref{eq:NonOpEqn}, we omit $F$
and denote the dependence of the forward solution on 
the uncertain input as $G(u)=q(u)$.
A {\em bounded linear observation operator} 
on the space $\cX$ of observed system responses in $Y$
is given and denoted by $\cO: \cX \rightarrow Y$.
Throughout the remainder of this paper,
we assume that $Y = \IR^K$ with $K<\infty$, then~$\cO\in \cL(\bcX; Y) \simeq (\cX^*)^K$.
The space $Y = \R^K$ is equipped with the Euclidean norm, denoted by $|\cdot|$.

We assume the following form of observed data,
composed of the observed system response
and the additive noise $\eta$
\begin{equation} \label{eq:DatDelta}
    \delta = \cO(G(u)) + \eta \,\in Y \;.
\end{equation}
The additive observation
noise process $\eta$ is assumed to be Gaussian, 
i.e.~a random vector $\eta \sim \bbQ_0 = \cN(0,\Gamma)$
with a positive definite covariance $\Gamma$ on $\R^K$.

The {\em uncertainty-to-observation map} 
$\cG:X \to \R^K$ then reads $\cG=\cO \circ G$, 
so that
\begin{equation}\label{eq:delta}
    \delta 
    = 
    \cG(u) + \eta = \cO(G(u)) + \eta. 
\end{equation}
When $\eta$ varies randomly, 
$\delta$ varies in $L^2_\Gamma(\R^K)$, the space of random vectors taking
values in $Y=\R^K$ which are square integrable with respect
to the Gaussian measure with covariance matrix $\Gamma$.
Bayes' formula from \cite{Stuart10,Stuart13} 
implies, for the $s$-variate parametric problems,
existence of a density of the Bayesian posterior with respect to the prior. 
Its negative log-likelihood equals the 
observation noise covariance-weighted, 
least squares functional (``potential'')
$\Phi_\Gamma:X \times Y \to \bbR$ 
by 
\begin{equation}\label{eq:DefPhiu}
\Phi_\Gamma(u;\delta)
=
\frac12|\delta - \cG(u) |_{\Gamma}^2
:=
\frac12
\left(
(\delta - \cG(u))^\top \Gamma^{-1} (\delta - \cG(u))
\right)
\;.
\end{equation}
%
%

\subsection{Well-posedness and approximation}
\label{sec:WellPosAppr}
In deterministic problems,
the data-to-solution maps is often ill-posed.
In the presently considered setting,
for a positive definite covariance $\Gamma$,
the expectations \eqref{eq:intpsi} are Lipschitz continuous
with respect to the data $\delta$,
{\em provided that the potential $\Phi_\Gamma$
     in \eqref{eq:DefPhiu} is locally
     Lipschitz with respect to the data $\delta$}. 
\begin{assumption}\label{Ass:LipPhi} \cite[Assumption 4.2]{Stuart13}
Let $\tilde{X}\subseteq X$ 
{\bf 
denote a subset of admissible uncertain input data,
}
and assume that 
$\Phi_\Gamma \in C(\tilde{X}\times Y;\IR)$
is Lipschitz on bounded sets.

Assume also that there exist
functions $\calM_i: \IR_+\times \IR_+\to \IR_+$, $i=1,2$,
(depending on $\Gamma > 0$)
which are monotone, non-decreasing separately in each argument,
and with $\calM_2$ strictly positive, 
such that for all $u\in \tilde{X}$,
and for all $\delta,\delta_1,\delta_2\in B_Y(0,r)$
\begin{equation}\label{eq:PhiBddbelow}
\Phi_{\Gamma}(u;\delta) \geq -\calM_1(r,\| u \|_X),
\end{equation}
and
\begin{equation}\label{eq:PhiLocLip}
|\Phi_\Gamma (u;\delta_1) - \Phi_\Gamma (u;\delta_2) | 
\leq 
\calM_2(r, \| u \|_X) | \delta_1 - \delta_2 | 
\;.
\end{equation}
\end{assumption}
In \cite[Thm.~4.3]{Stuart13}, a version of
Bayes' rule in function spaces is considered 
and the following result is shown.
\begin{prop}\label{thm:Bayes}
Let Assumption~\ref{Ass:LipPhi} hold.
Assume further that $\bpi(\tilde{X}) = 1$ and that $\bpi(\tilde{X}\cap B) > 0$
for some bounded set $B\subset X$. Assume additionally that, for every
fixed $r>0$,
\[
 \exp( \calM_1(r,\|u\|_X) ) \in L^1_{\bpi}(X;\R)\;.
\]
Then 

\noindent
(i) 
for $\bbQ_0$-a.e.~data $\delta\in Y$,
$$
Z:= \int_X\! \exp\left( -\Phi_\Gamma(u;\delta) \right) \bpi(\dd u) > 0 \;,
$$

\noindent
(ii) 
the conditional distribution of $u|\delta$ ($u$ given $\delta$) exists and
is denoted by $\bpi^\delta$.
It is absolutely continuous with respect to $\bpi$ and there holds
\begin{equation}\label{eq:BayesFormula}
\frac{d\bpi^\delta}{d\bpi}(u) 
=
\frac{1}{Z}  \exp\left( -\Phi_\Gamma(u;\delta) \right) 
\;.
\end{equation}
The Bayesian posterior admits a derivative w.r. to the prior 
with a bounded density w.r.t. the prior $\bpi$, which is given by \eqref{eq:BayesFormula}.
\end{prop}
%
We remark that in the present context,
\eqref{eq:PhiLocLip} follows from \eqref{eq:DefPhiu};
for convenient reference, we include \eqref{eq:PhiLocLip}
in Assumption \ref{Ass:LipPhi}. 
Under Assumption \ref{Ass:LipPhi},
the expectation \eqref{eq:intpsi} depends 
Lipschitz continuously on the data $\delta$ 
(see, e.g., \cite[Sec. 4.1]{Stuart13} for a proof):
\begin{equation} \label{eq:LipData}
\forall \phi\in L^2(\bpi^{\delta_1},X;\IR)\cap L^2(\bpi^{\delta_2},X;\IR):
\quad 
\| \bbE^{\bpi^{\delta_1}}[\phi] -  \bbE^{\bpi^{\delta_2}}[\phi]\|_\cZ
\leq 
C(\Gamma,r) 
  | \delta_1 - \delta_2 | \;.
\end{equation}
For $\delta \in Y$, 
the Bayesian posterior $\bpi^\delta_{h}$
with respect to the approximate potential $\Phi^{\ell}_\Gamma = \Phi^{(s_\ell, h_\ell)}_\Gamma$ 
obtained from dimension truncation at dimension $s_\ell$ and PG discretization \eqref{eq:parmOpEqh_trun}, 
with $O(h_\ell^{-d})$ degrees of freedom, is well-defined.
The multi-level convergence analysis requires bounds on the
approximation errors in the response due to
discretization and approximate numerical solution 
of the discretized problems) on the Bayesian estimate \eqref{eq:intpsi}.
To bound the errors of the Bayesian expectations \eqref{eq:intpsi} 
as in \cite{DGLGCSBIPSL}, we work under
\begin{assumption}\label{Ass:ApprPhi}
\cite[Assumption 4.7]{Stuart13}
Let $\tilde{X}\subseteq X$ denote a subset of admissible uncertain input data,
assume that $\bpi(\tilde{X}) = 1$, 
and let the Bayesian potential $\Phi_\Gamma \in C(\tilde{X};\IR)$
be Lipschitz on bounded sets.
Assume also that there exist functions 
$\calM_i: \IR_+ \to \IR_+$, $i=1,2$,
independent of the number $M = M_h = O(h^{-d})$ of degrees of freedom
in the PG discretization of the forward problem, where the functions $\calM_i$
are monotonically non-decreasing separately in each argument,
and with $\calM_2$ strictly positive,
such that for all $u\in \tilde{X}$ and for all $\delta \in B_Y(0,r)$,
\begin{equation}\label{eq:PhiBddbelow2}
\Phi_{\Gamma}(u;\delta) \geq - \calM_1(\| u \|_X),
\end{equation}
and there is a positive, monotonically decreasing $\varphi(\cdot)$
such that $\varphi(M)\to 0$ as $M\to \infty$,
monotonically and uniformly w.r.t. $u\in \tilde{X}$ (resp. w.r.t. $\bsy \in U$)
and such that
\begin{equation}\label{eq:PhiAppr}
|\Phi_\Gamma(u;\delta) - \Phi^{M}_\Gamma (u;\delta) | 
\leq 
\calM_2(\| u \|_X) \varphi(M)
\;.
\end{equation}
\end{assumption}
\begin{prop}\label{prop:ErrBayesExpec}
Suppose that Assumption \ref{Ass:ApprPhi} holds, and
assume that for $\tilde{X} \subseteq X$ and for some bounded
set $B\subset X$ we have $\bpi(\tilde{X} \cap B)>0$ 
and
$$
X\ni u \mapsto \exp(\calM_1(\| u\|_X)) (1 +\calM_2^2(\| u\|_X) ) \in L^1_{\bpi}(X;\IR) 
\;.
$$
Then there holds,
for every QoI $\phi:X\to \cZ$ such that
$\phi\in  L^2_{\bpi^{\delta}}(X;\cZ)\cap L^2_{\bpi^{\delta}_M}(X;\cZ)$
uniformly w.r.t. $M$ and such that $Z>0$ in \eqref{eq:Z},
the consistency error bound
\begin{equation}\label{eq:BayAppr}
\|\bbE^{\bpi^{\delta}}[\phi] -  \bbE^{\bpi^{\delta}_M}[\phi]\|_\cZ
\leq 
C(\Gamma,r) 
\varphi(M)
\;.
\end{equation}
\end{prop}
%
A proof of Proposition \ref{prop:ErrBayesExpec}
can be found in \cite[Thm.~4.9, Rem.~4.10]{Stuart13}.
%
%
The consistency condition \eqref{eq:PhiAppr} in either of these cases
follows from \cite[Prop.~3.2]{DGLGCSBIPSL}.
Assume we are given a sequence $\{ q^M \}_{M \geq 1}$
of approximations to the forward response
$X\ni u\mapsto q(u) \in \bcX$
such that, with the parametrization \eqref{eq:uviapsi},
\begin{equation}\label{eq:qConsis}
\sup_{u\in \tilde{X}} \| (q-q^M)(\bsy) \|_{\bcX} \leq \varphi(M)
\end{equation}
with a consistency error bound $\varphi(M) \downarrow 0$ as
in Assumption \ref{Ass:ApprPhi}.
Denote by $G^M$ the corresponding
(Petrov-Galerkin) approximations of the parametric forward maps.
Then it was shown in \cite[Prop.~3.2]{DGLGCSBIPSL} that 
the approximate Bayesian potential
\begin{equation}\label{eq:PhiN}
\Phi^M_{\Gamma}(u;\delta) 
= \frac{1}{2} (\delta - \cG^M(u))^\top \Gamma^{-1} (\delta - \cG^M(u))
: X\times Y \to \bbR \, ,
\end{equation}
where $\cG^M := \cO\circ G^M$, satisfies \eqref{eq:PhiAppr}.
In the ensuing error analysis of the combined HoQMC, PG 
approximation, $q^M$ in \eqref{eq:PhiN} will denote the
Petrov-Galerkin discretization \eqref{eq:parmOpEqh_trun}
in Section \ref{sec:Discr}, 
based on $s$-dimensional truncation of the parameter space, 
and on finite dimensional subspaces with 
$M = M_h = {\rm dim}(\bcX^h) = {\rm dim}(\bcY^h)$ many degrees of freedom.
Typically, for subspaces of piecewise polynomial functions
obtained by isotropic mesh refinement in a bounded, 
physical domain $D\subset \IR^d$, $M_h = O(h^{-d})$; 
in sparse grid discretizations in $D$ (not considered in detail here,
but covered by the present theory), $M_h = O(h^{-1}|\log h|^{d-1})$.
In that case, $\bcX_t$ and $\bcY_t$ are 
function spaces of mixed partial 
derivatives.

In applying QMC to Bayesian estimation, we shall be concerned with
integration of parametric integrand functions with respect to 
a probability measure $\bpi$ on the set $U$ of all parameters characterizing
the uncertain input data. Typically, $\bpi$ will denote a Bayesian
prior on $U$. Since $U$ is a cartesian product of intervals,
we \emph{assume} that $\bpi$ is a product probability measure, i.e.~
$\bpi(\mathrm{d}\bsy) = \prod_{j\geq 1} \pi_j(\mathrm{d}y_j)$.
\subsection{Parametric Bayesian posterior}
\label{eq:ParaPost}

As in \cite{SS12,SS13,DGLGCSBIPSL} and in Section \ref{sec:Param},
we adopt a product probability space
\begin{equation}\label{eq:UcBmu0}
(U,\cB,\bpi) \;.
\end{equation}
{\em 
We assume that the prior $\bpi$ on the uncertain input $u\in X$, 
parametrized in the form \eqref{eq:uviapsi}, 
satisfies $\bpi(\tilde{X}) = 1$.
}
We also assume the prior measure $\bpi$ being uniform, and the sequences $\bsy$
in \eqref{eq:uviapsi} taking values in the parameter domain $ U=[-1,1]^\bbN$.
With $U$ as in (\ref{eq:UcBmu0}),
the uncertainty-to-observation map $\Xi:U \to Y = \bbR^K$ reads
\begin{equation}  \label{eq:DefMapXi}
\Xi(\bsy)
=
\cG(u)\Bigl|_{u= \langle u \rangle+\sum_{j \in \mathbb{N} } y_j\psi_j}
\;.
\end{equation}
In order to apply our QMC quadrature, 
we need a parametric version of Bayes' Theorem, as stated in Prop.~\ref{thm:Bayes},
in terms of the uncertainty parametrization \eqref{eq:uviapsi}.
We view $U$ as the unit ball in $\ell^\infty([-1,1]^{\bbJ}$,
the Banach space of bounded sequences taking values in $U$.
\begin{prop} \label{t:dens}
Assume that $\Xi: U \to Y = \bbR^K$ is bounded and continuous. 
Then $\bpi^\delta(\dd\bsy)$, the distribution of 
$\bsy\in U$ given data $\delta\in Y$, 
is absolutely continuous with respect to $\bpi(\dd\bsy)$, i.e.
there exists a parametric density $\Theta(\bsy)$ such that
for every $\bsy\in U$
\be \label{eq:post}
\frac{d\bpi^\delta}{d\bpi}(\bsy) = \frac{1}{Z} \Theta(\bsy)
\ee
with the posterior density $\Theta(\bsy)$ given by
\begin{equation} \label{eq:PostDens}
\Theta(\bsy) 
= 
\exp\bigl(-\Phi_\Gamma(u;\delta)\bigr)\Bigl|_{u=\langle u \rangle + \sum_{j \in \bbJ} y_j\psi_j} 
\;.
\end{equation}
Here the Bayesian potential $\Phi_\Gamma$ is as in \eqref{eq:DefPhiu} and 
\be \label{eq:Z}
Z 
= \IE^{\bpi}\!\left[ 1 \right] 
= \int_{U}\! \Theta(\bsy)\,
\bpi(\dd \bsy) > 0 \;.
\ee
\end{prop}
Bayesian estimation 
aims at a ``most likely'' value of the Quantity of Interest (QoI) 
$\phi:X \to \cZ$ (in a Banach space $\cZ$),
conditional on given (noisy) observation data $\delta\in Y$.
For example, $\phi(u) = G(u) = q(u)$ (with $\cZ=\bcX$)
will result in a
``most likely'' (as expectation under the posterior,
 given data $\delta$) system response.
Based on Prop.~\ref{t:dens} we associate
with the QoI $\phi$ the deterministic, infinite-dimensional, parametric map
\begin{equation}
\begin{array}{rcl}
\Psi(\bsy)
&=&
\Theta(\bsy)\phi(u) \mid_{u= \langle u \rangle + \sum_{j \in \bbJ} y_j\psi_j}
=
\exp\bigl(-\Phi_\Gamma(u;\delta)\bigr)\phi(u)
\Bigl|_{ u = \langle u \rangle + \sum_{j \in \bbJ} y_j\psi_j }
: U\rightarrow \cZ 
\;.
\label{eq:psi}
\end{array}
\end{equation}
With the density $\Psi(\bsy)$ in \eqref{eq:psi},
the Bayesian estimate of the QoI $\phi$, 
given noisy data $\delta$, takes the form
\begin{equation}\label{eq:intpsi}
\IE^{\bpi^\delta}[\phi]
= 
Z'/Z, \;\;
\quad 
Z':=
\IE^{\bpi}[\Psi] 
=
\int_{U} \! \Psi(\bsy) \,\bpi(\dd \bsy)
\;.
\end{equation}
%
%
\subsection{Higher order QMC integration}
\label{sec:HoQMC}
In general, the integrals $Z$ and $Z'$ in \eqref{eq:intpsi} 
cannot be evaluated exactly. 
Hence we propose to approximate these integrals by a HoQMC quadrature rule.
In the following, we assume the prior density $\bpi$ to be the uniform density.
Supposing $\{\bsy_0, \ldots,\bsy_{N-1} \} \subset U=[0,1]^s$ 
is a collection of $N$ QMC points, we define the QMC estimate
of the integral of a function $f:U\to\cZ$ with respect to the uniform
measure $\bpi$ by the following equal weight quadrature rule in $s$ dimensions,
\begin{equation}\label{eq:QMCInt}
    Q_{N,s} [f]
    :=
    \frac{1}{N} 
    \sum_{j=0}^{N-1} f(\bsy_j)
    \;.
\end{equation}
We shall analyze, in particular, $Q_{N,s}$ being 
\emph{deterministic, interlaced higher order polynomial lattice rules}
as introduced in \cite{D08} and as considered for affine-parametric 
operator equations in \cite{DKGS14}.

To generate a polynomial lattice rule with $N=b^m$ points (where $b$ is a given
prime number, $m$ is a given positive integer), we need a \emph{generating vector} 
of polynomials 
$
  \bg(x) = (g_1(x), \ldots, g_s(x))
$
where each $g_j(x)$ is a polynomial with degree $<m$ with its coefficients 
taken from a finite field ${\mathbb Z}_b$.

For each integer $n= 0,\ldots,b^m-1$, we associate $n$ with the polynomial
\[
 n(x) = \sum_{r=0}^{m-1} \eta_r x^r  \quad \in {\mathbb Z}_b[x],
\]
where $(\eta_{m-1}, \ldots ,\eta_0)$ is the $b$-adic expansion of $n$,
that is $n = \eta_0 + \eta_1 b + \ldots + \eta_{m-1} b^{m-1}$.
We also need a map $v_m$ which maps elements in ${\mathbb Z}_b(x^{-1})$ to
the interval $[0,1)$, defined for any integer $w$ by
\[
  v_m \left( \sum_{\ell=w}^\infty t_{\ell} x^{-\ell} \right) =
  \sum_{\ell=\max(1,w)}^m t_{\ell} b^{-\ell}.  
\]

Let $P \in {\mathbb Z}[x]$ be an irreducible polynomial with degree $m$.
The classical polynomial lattice rule $\calS_{P,b,m,s}(\bg)$ associated with 
$P$ and the generating vector $\bg$ is comprised of the quadrature points
\[
 \bsy_n = \left( v_m \left( \frac{n(x)g_1(x)}{P(x)} \right), 
       \ldots, 
      v_m\left( \frac{n(x) g_s(x)}{P(x)} \right)  \right) \in [0,1)^s,
\quad n = 0,\ldots, b^{m} - 1.
\]

Classical polynomial lattice rules give almost first order of convergence for integrands of bounded variation.
To obtain higher order of convergence, an interlacing procedure described as follows
is needed. The \emph{digit interlacing function} with digit interlacing factor 
$\alpha \in \mathbb{N}$ is given by
\begin{equation}\label{eq:DigIntl}
\begin{array}{rcl}
\mathscr{D}_\alpha: [0,1)^{\alpha} & \to & [0,1) \;: \;
(x_1,\ldots, x_{\alpha}) \mapsto \sum_{a=1}^\infty \sum_{j=1}^\alpha
\xi_{j,a} b^{-j - (a-1) \alpha}\;,
\end{array}
\end{equation}
where $x_j = \xi_{j,1} b^{-1} + \xi_{j,2} b^{-2} + \cdots$
for $1 \le j \le \alpha$.
For vectors, we set
\begin{equation} \label{eq:DigIntlAr}
\begin{array}{rcl}
\mathscr{D}_\alpha: [0,1)^{\alpha s} & \to & [0,1)^s
\\
(x_1,\ldots, x_{\alpha s})
&\mapsto &
(\mathscr{D}_\alpha(x_1,\ldots, x_\alpha),  \ldots,
\mathscr{D}_\alpha(x_{(s-1)\alpha +1},\ldots, x_{s \alpha}))
\;.
\end{array}
\end{equation}

Then, an \emph{interlaced polynomial lattice rule of
order $\alpha$ with $b^m$ points in $s$ dimensions} 
is a QMC rule using $\mathscr{D}_\alpha(\calS_{P,b,m,\alpha s}(\bg))$ 
as quadrature points, for some given modulus $P$ and generating vector $\bg$.
\subsection{Holomorphic parameter dependence}
\label{sec:Hol}
In the QMC quadrature error analysis, for each parameter 
$\bsy_j$, for $j=1,\ldots,N$, we assume the parametric 
family of solutions admits a holomorphic extension into the complex
domain $\C$ and these extensions must satisfy some uniform bounds.
We recall from \cite{HaSc11,CCS2} the notion of 
{\em $(\bsb, p, \eps)$-holomorphy of parametric solutions}
introduced to this end. %
\emph{%
In the remainder of Section~\ref{sec:Hol} 
all spaces $X$, $Y$, $\bcX$ and $\bcY$ will be 
assumed to be Banach spaces over $\C$. 
}
\begin{definition} \label{def:peanalytic} ($(\bsb,p,\eps)$-holomorphy)
Let $\eps > 0$ and $0<p<1$ be given. 
For a positive sequence 
$\bsb = (b_j)_{j\geq 1} \in \ell^p(\IN)$, 
a parametric mapping $g: U\to \cX$ satisfies the {\em $(\bsb,p,\eps)$-holomorphy assumption} 
if and only if all of the following conditions are satisfied:
\begin{enumerate}
\item 
The map $\bsy\mapsto g(\bsy)$ from $U$ to $\cX$, for each $\bsy\in U$,
is uniformly bounded with respect to the parameter sequence $\bsy$, 
i.e.~there is a bound $C_0 > 0$ such that
\be
\label{ubNewProblem}
\sup_{\bsy\in U}\|g(\bsy)\|_\bcX \leq C_0\;.
\ee
\item 
%
%
For any \emph{$(\bsb,\eps)$-admissible} sequence $\bsrho:=(\rho_j)_{j\geq1}$,
for sufficiently small $\eps > 0$, 
the parametric solution 
$\bsy\mapsto g(\bsy)$ 
admits a continuous extension
$\bsz \mapsto g(\bsz)$ to the complex domain
with respect to each variable $z_j$
that is holomorphic in a cylindrical set of the form 
${O}_{\bsrho} := \bigotimes_{j\geq 1} {O}_{\rho_j}$ 
where, for every $j\geq 1$, 
$[-1,1] \subset O_{\rho_j}\subset \C$ 
is an open set  $O_{\rho_j}\subset\C$.
Recalling that, the sequence $\bsrho:=(\rho_j)_{j\geq1}$ is \emph{$(\bsb,\eps)$-admissible}
if
\be\label{eq:rho_b}
\sum_{j \geq 1} (\rho_j-1) b_j \leq \eps,
\ee

\item
For any poly-radius $\bsrho$ satisfying \eqref{eq:rho_b}, 
there are open sets 
$$
[-1,1] \subset {O}_{\rho_j} \subset \tilde{O}_{\rho_j} \subset \C
$$
such that the parametric map admits a continuous 
extension to $\overline{\tilde{O}_\bsrho}$ which is bounded:
there holds
\be
\sup_{\bsz\in\tilde{O}_\rho}\|g(\bsz)\|_\bcX \leq C_\eps \, ,
\ee
where the bounds $C_\eps > 0$ depend on $\eps$, 
but are independent of $\bsrho$.
\end{enumerate}
\end{definition}
Evidently, 
$(\bsb, p,\eps)$-holomorphy depends on the choice of $\tilde{O}_{\rho_j}$.
For derivative estimates in the HoQMC error bounds in 
\cite{DKGNS13,KSS12}, we use as in \cite{DLGCS14}
a specific family of continuation domains:
for $\kappa>1$, define
%
\begin{equation}\label{eq:TubeDef}
\cT_\kappa 
= \{ z\in \bbC| {\rm dist}(z,[-1,1]) \leq \kappa - 1\}
= \bigcup_{-1\leq y \leq 1} \{ z\in \IC | |z-y|\leq \kappa -1 \}
\;.
\end{equation}
Then, once more, for a poly-radius $\bsrho$ satisfying
\eqref{eq:rho_b}, we denote by $\cT_\bsrho$ the 
corresponding cylindrical set 
$\cT_\bsrho := \bigotimes_{j\geq 1} \cT_{\rho_j}\subset \C^\N$.
We refer to \cite{HaSc11,CCS2} and the references there 
for further examples of parametric forward problems  whose solutions admit 
$(\bsb,p,\eps)$-holomorphic extensions.
\subsection{HoQMC convergence for $(\bsb,p,\eps)$-holomorphic integrands}
\label{sec:anadepsol}
To obtain QMC error bounds, 
we quantify the dependence of the parametric
solution $\bsy \mapsto q(\bsy)\in \bcX$ of \eqref{eq:NonOpEqn} 
on the parameter sequence $\bsy$.
As shown in  \cite{KSS12},
bounds on the growth of the partial derivatives 
of $q(\bsy)$ with respect to $\bsy$ imply
convergence rates for higher order QMC quadratures.

We are interested in 
integrand functions $g$ being compositions of
$\Oc(\cdot) \in \bcX'$
with PG approximation $q_h^s(2 \bsy-\bsone)$ 
of the dimension-truncated, parametric and
$(\bsb,p,\eps)$-holomorphic, operator equation \eqref{eq:NonOpEqn}.  
For every $s\in N$, $g(\bsy) := (\Oc \circ q^s)(\bsy_{\{1:s\}})$ 
is $(\bsb,p,\eps)$-holomorphic {\em uniformly w.r.t. $s\in \IN$}
(see Section \ref{sec:dimtrunc}).
Based on 
\cite[Sec.~3]{DKGNS13} and \cite[Prop.~4.1]{DLGCS14},
the following result provides rates of convergence of QMC quadratures.
\begin{prop}\label{prop:main1}
Assume given a parameter space dimension 
$s\ge 1$ and $N = b^m$ integration points
for integer $m\ge 1$ and for a prime number $b$. 
Let $\bsbeta = (\beta_j)_{j\ge 1}$ be a sequence of positive numbers, 
and denote by $\bsbeta_s = (\beta_j)_{1\le j \le s}$ its truncation after $s$ terms. 
Assume that $\bsbeta \in \ell^p(\IN)$ for some $0<p<1$, i.e.~that
\begin{equation} \label{p-sum}
  \exists\, 0<p < 1 : \quad \sum_{j=1}^\infty \beta_j^p < \infty\;.
\end{equation}
Define, for $0<p\leq 1$ as in \eqref{p-sum}, the digit interlacing parameter
\begin{equation} \label{alpha}
  \alpha \,:=\, \lfloor 1/p \rfloor +1 \;.
\end{equation}
Consider parametric integrand functions $g: U \to \cZ$ 
in a separable Hilbert space $\cZ$
which are $(\bsbeta,p,\eps)$-holomorphic 
(cp. Definition \ref{def:peanalytic}).

Then, for every $N\in \N$, one can construct 
an interlaced polynomial lattice rule of
order $\alpha$ with $N$ points using a fast
component-by-component algorithm, 
in at most $\calO(\alpha^2 s N \log N)$ operations, 
plus $\calO(\alpha^2 s^2 N)$ update cost, plus 
$\calO(\alpha s N)$ memory cost, 
and with the error bound
\begin{equation}\label{qmc_upper_bound}
\forall s,N \in \N:\quad 
  \|I_s(g) - Q_{N,s}(g)\|_{\cZ}
  \,\le\, C_{\alpha,\bsbeta,b,p}\, \| g \|_{\bsrho,\cZ} N^{-1/p} \;.
\end{equation}
Here $C_{\alpha,\bsbeta,b,p} < \infty$ is independent of $s$
and $N$, and  the expression 
$\| g\|_{\bsrho,\cZ}$ denotes the 
(likewise independent of $s$ and of $N$)
maximum of the modulus for the 
integrand function over $\cT_\bsrho$ for every 
$(\bsbeta,\eps)$-admissible poly-radius
(i.e., satisfying \eqref{eq:rho_b}) $\bsrho = (\rho_j)_{j\geq 1}$:
\be\label{eq:TrhoNZ}
\| g \|_{\bsrho,\cZ} := \sup_{\bsz \in \cT_{\bsrho}} \| g(\bsz) \|_\cZ 
\;.
\ee
\end{prop}
\begin{proof}
The Hilbert space $\cZ$ being isomorphic to $\cZ'$,
for every $g \in \cZ$ the norm $\| g \|_{\cZ}$ can be written as
$$
\| g \|_{\cZ}
= 
\sup_{\chi \in \cZ: \| \chi \|_\cZ \leq 1} |(\chi, g)_\cZ| 
\;,
$$
where $(\cdot,\cdot)_\cZ$ denotes the $\cZ$ inner product.
By separability of $\cZ$, the Bochner-integral $I_s[g]$ and 
the QMC rules $Q_{N,s}[g]$ are well-defined also for 
measurable integrand functions taking values in $\cZ$.
For integrand functions 
$g$ being a $(\bsb,p,\eps)$-holomorphic map taking values
in $\cZ$, for every $\chi\in \cZ$ the mapping $\bsz \mapsto (\chi,g(\bsz))_\cZ$
is a $(\bsb,p,\eps)$-holomorphic map with values in $\IC$.
Therefore, the complex variable HoQMC error bounds obtained in
\cite[Prop.~4.1]{DLGCS14} apply for this integrand: 
there exists a constant 
$C_{\alpha,\bsbeta,\bsb,p}>0$ which is independent of the 
parametric integrand
$\bsy \mapsto (\chi,g(\bsy))_\cZ$ 
such that for every $N$ holds the 
HoQMC error bound
\be\label{eq:QMCei}
\left| (I_s - Q_{N,s})\left[ (\chi ,g(\cdot))_\cZ  \right] \right|
\leq
C_{\alpha,\bsbeta,b,p}\, \|  (\chi,g(\cdot))_\cZ  \|_{\bsrho,\IC} N^{-1/p} 
\;.
\ee
Here, for every fixed $\chi$ and any  $(\bsbeta,\eps)$-admissible poly-radius $\bsrho$,
$$
\|  (\chi,g(\cdot))_\cZ  \|_{\bsrho,\IC}
=
\sup_{\bsz\in \cT_\bsrho} |  (\chi,g(\bsz))_\cZ | \;.
$$
The linearity of $I_s[\cdot]$ and of $Q_{N,s}[\cdot]$ imply 
that for every $\chi \in \cZ$ holds
$$
(I_s - Q_{N,s})\left[ (\chi,g(\cdot))_\cZ  \right]
=
\left(\chi, (I_s - Q_{N,s})\left[ g(\cdot)\right] \right)_\cZ
\;.
$$
We may therefore estimate with \eqref{eq:QMCei}
%
%
$$
\begin{array}{rcl}
\displaystyle 
\left\| (I_s - Q_{N,s})\left[ g(\cdot)\right] \right\|_\cZ
& = & \displaystyle 
\sup_{\chi\in \cZ: \| \chi \|_\cZ \leq 1} 
\left|
\left( \chi , (I_s - Q_{N,s})\left[ g(\cdot)\right] \right)_\cZ  
\right|
\\
& = & \displaystyle 
\sup_{\chi\in \cZ: \| \chi \|_\cZ \leq 1} 
\left|
(I_s - Q_{N,s})\left[ (g(\cdot) , \chi)_\cZ  \right]
\right|
\\
& \leq & \displaystyle 
C_{\alpha,\bsbeta,b,p}\, N^{-1/p} 
\sup_{\chi\in \cZ: \| \chi \|_\cZ \leq 1}
\|  (g(\cdot),\chi)_\cZ  \|_{\bsrho,\IC}
\\
& = & \displaystyle 
C_{\alpha,\bsbeta,b,p}\, N^{-1/p} 
\sup_{\chi\in \cZ: \| \chi \|_\cZ \leq 1}
\sup_{\bsz\in \cT_\bsrho} |  (g(\bsz),\chi)_\cZ |
\\
& = & \displaystyle 
C_{\alpha,\bsbeta,b,p}\, N^{-1/p} 
\sup_{\bsz\in \cT_\bsrho} 
\sup_{\chi\in \cZ: \| \chi \|_\cZ \leq 1} |  (g(\bsz),\chi)_\cZ |
\\
& = & \displaystyle 
C_{\alpha,\bsbeta,b,p}\, N^{-1/p} 
\sup_{\bsz\in \cT_\bsrho} 
\| g(\bsz) \|_\cZ \;.
\end{array}
$$
\end{proof}
\section{Multilevel HoQMC Bayesian estimators}
\label{sec:MLHOQMC}
\subsection{Multilevel discretization of the forward problem}
\label{sec:MLDiscr}
Multilevel Bayesian estimators are based on a 
\emph{hierarchy of discretizations} of the forward problem
\eqref{eq:NonOpEqn}.
Below we design 
\emph{increasing sequences $\{ s_\ell \}_{\ell \geq 0}$ and $\{ N_\ell \}_{\ell\geq 0}$ of
truncation dimensions and QMC sample numbers} at discretization level $\ell$,
and a 
\emph{decreasing sequence $\{ h_\ell \}_{\ell\geq 0}$ of discretization parameters}.
With this latter sequence, we associate 
\emph{nested, dense sequences 
$\{ \bcX_\ell \}_{\ell\geq 0}  = \{ \bcX^{h_{\ell}} \}_{\ell\geq 0}$
and
$\{ \bcY_\ell \}_{\ell\geq 0}  = \{ \bcX^{h_{\ell}} \}_{\ell\geq 0}$
}
of discretization spaces of 
increasing, finite dimension $M_\ell = {\rm dim}(\bcX_\ell) = {\rm dim}(\bcY_\ell)$. 
We refer to the pairs $(s_\ell, h_\ell )_{\ell \geq 0}$ as 
\emph{PG discretization parameters}, 
and also set $s_{-1} := 1$ and $h_{-1}:=1$.
We denote by 
$q^{(\ell)} = q^{h_\ell}(\bsy_{\{ 1: s_\ell \}}) \in \bcX_\ell$ 
the dimensionally truncated, PG approximated parametric forward solution 
in \eqref{eq:parmOpEqh_trun}.
The consistency error bounds for approximations of the Bayesian posterior, 
\eqref{eq:BayAppr}, imply that the Bayesian estimates incur an
error which is bounded by the error of the PG discretization in the forward problem,
\emph{provided posterior expectations are evaluated exactly}.

The potential $\Phi_\Gamma$ based on the solution of \eqref{eq:NonOpEqn} on 
discretization level $\ell$
is denoted as $\Phi^\ell_\Gamma$, and is given in \eqref{eq:PhiN}.
The corresponding approximations to the posterior density $\Theta(\bsy)$ 
from \eqref{eq:PostDens} are then
\begin{equation}\label{eq:Thetell}
    \Theta_\ell(\bsy)
    :=
    \exp(-\Phi^\ell_\Gamma(u(\bsy_{\{ 1: s_\ell \}});\delta))
    \;.
\end{equation}
The \emph{exact normalization constant}
for the forward model at discretization level $\ell$ is then
\begin{equation}\label{eq:Zell}
    Z_\ell := \int_ U \Theta_\ell(\bsy) \,\bpi(\dd \bsy)\;.
\end{equation}
Note that $\Theta_\ell(\bsy)$, as defined in \eqref{eq:Thetell}, 
includes also a dimension truncation to dimension $s_\ell$.

In our ensuing error analysis of the Bayesian estimators,
we shall use the bound  \eqref{qmc_upper_bound} for various choices of 
$g(\bsz)$; in particular, with
$g(\bsz) = (\phi(q)\Theta - \phi(q^{(\ell)})\Theta_\ell)(\bsz)$ 
both with $\phi \equiv 1$, $\cZ = \IR$ 
and with $\phi(\cdot):\bcX \to \bcZ$ denoting the QoI.

Given a QoI $\phi:\bcX \to \cZ$ and 
noisy observation data $\delta\in Y$ as in \eqref{eq:delta},
we approximate the Bayesian estimate $\bbE^{\bpi^\delta} [\phi]$ 
defined in \eqref{eq:intpsi} by the expectation 
\begin{equation}\label{eq:EpiL}
    \bbE^{\bpi^\delta_\ell} [\phi]
    =
    \frac{1}{Z_\ell}
    \int_U \phi(q^{(\ell)}(\bsy)) \Theta_\ell(\bsy) \,\bpi(\dd\bsy) 
    = \int_U \phi(q^{(\ell)}(\bsy)) \pi^{\delta}_{\ell} (\dd \bsy) 
    = \frac{Z_\ell'}{Z_\ell} \;,
\end{equation}
where $\pi^{\delta}_{\ell}$ is given by the Radon-Nikodym derivative 
$\dd \pi^{\delta}_{\ell} / \dd \pi = \Theta_\ell/ Z_\ell$.


\subsection{Multilevel Bayesian estimators}
\label{sec:BayEst}
We now describe two multilevel estimators, 
the ratio estimator proposed in \cite{SS12,SS13} 
and the splitting estimator, inspired by \cite{BJLTZ15}.
The ratio estimator is based on expanding the two integrals $Z_L'$ and $Z_L$ separately,
and the splitting estimator is based on a joint telescopic expansion of $Z_L'/Z_L$.
In conjunction with deterministic higher order QMC integration \eqref{eq:QMCInt},
either estimator will result in deterministic algorithms for Bayesian estimation
of PDEs with distributed, uncertain input with high convergence rates which are
independent of the dimension of the parameter space.
%
\subsubsection{Multilevel ratio estimator}
\label{sec:QuotEst}
It was proposed in \cite{SS12,SS13} to numerically evaluate
the Bayesian estimates \eqref{eq:intpsi} by ``direct integration'',
i.e.~by applying a deterministic quadrature rule to the (formally)
infinite-dimensional, iterated integrals in the normalization constant $Z$
and the QoI $Z'$. 
In \cite{SS12,SS13}, 
dimension adaptive Smolyak quadrature was used to compute $Z'/Z$.
Here, as in \cite{DGLGCSBIPSL}, we evaluate $Z_L'/Z_L$ 
by approximating the integrals $Z$ and $Z'$ by HoQMC integration.
In \cite{DGLGCSBIPSL} we used single-level estimators for numerator and 
denominator; here, we develop a MLHoQMC estimator. 
To derive it, we write, using the telescopic sum identity and 
denoting integration with respect to the prior $\bpi$ by $I[\cdot]$,
\begin{align}
\label{eq:TelZ'}
    Z_L'
    &= 
    Z_0' + \sum_{\ell=1}^L (Z_\ell'-Z_{\ell-1}') 
    =
    I[ \phi(q^{(0)}) \Theta_0 ]
    +
    \sum_{\ell=1}^L
    I\Big[
        \phi(q^{(\ell)}) \Theta_\ell - \phi(q^{(\ell-1)}) \Theta_{\ell-1}
    \Big] 
    \\
    &= \nonumber
    \int_U \phi(q^{(0)}(\bsy)) \Theta_0(\bsy) \,\bpi(\dd\bsy) +
    \sum_{\ell=1}^L
    \int_U \Big(
        \phi(q^{(\ell)}(\bsy)) \Theta_\ell(\bsy)
        -
        \phi(q^{(\ell-1)}(\bsy)) \Theta_{\ell-1}(\bsy)
    \Big) \,\bpi(\dd\bsy) \;,
\end{align}
and the corresponding ML approximation of the normalization constant $Z$, 
\begin{align} \label{eq:TelZ}
    Z_L &= 
           Z_0 + \sum_{\ell=1}^L (Z_\ell-Z_{\ell-1}) 
          = I[ \Theta_0 ] + \sum_{\ell=1}^L I\Big[ \Theta_\ell - \Theta_{\ell-1} \Big] 
    \\ 
    &= \nonumber
    \int_U \Theta_0(\bsy) \,\bpi(\dd\bsy)
    +
    \sum_{\ell=1}^L
    \int_U \Big(
        \Theta_\ell(\bsy) - \Theta_{\ell-1}(\bsy)
    \Big) \,\bpi(\dd\bsy)
    \;.
\end{align}
We approximate the integrals in \eqref{eq:TelZ'} and \eqref{eq:TelZ} 
on discretization level $\ell=0,\ldots,L$ with 
a higher order QMC rule based on $N_\ell$ points,
resulting in the Multilevel Higher Order QMC estimators $Q_L^\star[Z']$ 
and $Q_L^\star[Z]$, respectively.
Note that both estimators (for $Z$ and $Z'$) can be computed
simultaneously, without having to reevaluate the forward model.
The \emph{multilevel ratio estimator} is thus given by
\begin{equation}\label{eq:QLratio}
    Q^\star_{L,ratio} := Q_L^\star[Z'] / Q_L^\star[Z]
    \;,
\end{equation}
with the individual multilevel approximations (omitting the variables $\bsy$) 
given by the \emph{multilevel algorithm $Q_L^\star$ of \cite{KSS13,DKGS14}}, 
i.e.
\begin{align}\label{eq:MLratioZ'}
    Z_L'
    &\approx
    Q_L^\star[Z']
    :=
    Q_{N_0,s_0}[ \phi(q^{(0)}) \Theta_0 ]
    +
    \sum_{\ell=1}^L
    Q_{N_\ell,s_\ell} \Big[
        \phi(q^{(\ell)}) \Theta_\ell - \phi(q^{(\ell-1)}) \Theta_{\ell-1}
    \Big] \;,\\
    Z_L
    &\approx
    Q_L^\star[Z]
    :=
    Q_{N_0,s_0}[ \Theta_0 ]
    +
    \sum_{\ell=1}^L
    Q_{N_\ell,s_\ell} \Big[ \Theta_\ell - \Theta_{\ell-1} \Big]
    \;.
    \label{eq:MLratioZ}
\end{align}
\subsubsection{Multilevel splitting estimator}
\label{sec:SplitEst}
We consider approximations of the form \eqref{eq:EpiL} on each 
discretization level $\ell=0,\ldots,L$ and write 
\begin{equation}\label{eq:BIsplit}
    \frac{Z_L'}{Z_L}
    =
    \frac{Z_0'}{Z_0}
    +
    \sum_{\ell=1}^L
    \left(
        \frac{Z_\ell'}{Z_\ell} - \frac{Z_{\ell-1}'}{Z_{\ell-1}}
    \right)
    \;.
\end{equation}
We now apply QMC quadrature to the integrals in $Z_\ell$, $Z_\ell'$
for each $\ell=0,\ldots,L$, where we approximate all terms on 
the same discretization level $\ell$ (i.e.~all terms inside the parenthesis)
by a QMC rule with $N_\ell$ points, resulting in the 
Multilevel High Order QMC estimator $Q_L^\ast$.
Denoting by $Z_{\ell,k} = Q_{N_k}[Z_\ell]$ and $Z_{\ell,k}' = Q_{N_k}[Z_\ell']$
the HoQMC approximations of $Z_\ell$ and $Z_\ell'$ on quadrature level $k$,
the \emph{multilevel QMC splitting estimator} 
reads
\begin{equation}\label{eq:QLsplit}
    Q^\star_{L,split}
    :=
    \frac{Z_{0,0}'}{Z_{0,0}}
    +
    \sum_{\ell=1}^L
    \left(
        \frac{Z_{\ell,\ell}'}{Z_{\ell,\ell}}
        -
        \frac{Z_{\ell-1,\ell}'}{Z_{\ell-1,\ell}}
    \right)
    \;.
\end{equation}
Thus, on discretization level $\ell=0$ we must compute the two approximations
$Z_{0,0} = Q_{N_0}[\Theta_0]$ and $Z_{0,0}'=Q_{N_0}[\phi(q^0) \Theta_0]$,
and for each discretization level $\ell=1,\ldots,L$ we 
evaluate four approximations
\begin{align}\label{eq:fourintegrals}
    Z'_{\ell,\ell}   &= Q_{N_\ell,s_\ell}\big[ \phi(q^{(\ell)}(\bsy))\Theta_\ell(\bsy) \big] &
    Z'_{\ell-1,\ell} &= Q_{N_\ell,s_\ell}\big[ \phi(q^{(\ell-1)}(\bsy))\Theta_{\ell-1}(\bsy) \big] 
    \nonumber\\
    Z_{\ell,\ell}    &= Q_{N_\ell,s_\ell}\big[ \Theta_\ell(\bsy) \big] &
    Z_{\ell-1,\ell}  &= Q_{N_\ell,s_\ell}\big[ \Theta_{\ell-1}(\bsy) \big]
    \;.
\end{align}
Since all four approximations involve the same QMC quadrature points,
the additional quadratures do not require 
extra solutions of the forward model on level $\ell$ or $\ell-1$.
Therefore, they do not increase the work significantly,
assuming the cost of representing elements of $\cZ$ in an implementation are negligible.
In the case where different truncation dimensions $s_\ell>s_{\ell-1}$ are used,
the ``quadrature point'' $\bsy_{s_\ell}$ on level $\ell$ 
can be truncated at dimension $s_{\ell-1}$:
$\bsy_{s_{\ell-1}} := (\bsy_{s_\ell})_{\{1:  s_{\ell-1}\}} 
  \in[-\frac12,\frac12]^{s_{\ell-1}}$ on level $\ell-1$.

\begin{remark}\label{rmk:Splt}
The differences in the formula \eqref{eq:BIsplit} can be written 
as the following expression (bearing in mind that index $\ell$ 
signifies implied dimension truncation to dimension $s_\ell$)
\begin{align}\label{eq:DiffEx}
    \frac{Z_{\ell}'}{Z_{\ell}}
    -
    \frac{Z_{\ell-1}'}{Z_{\ell-1}}
    & =
    \int_{\bsy\in [-1/2,1/2]^{s_\ell}}
    \left(
    \phi(q^{(\ell)}(\bsy))
    -
    \frac{\Theta_{\ell-1}(\bsy)Z_\ell}{\Theta_\ell(\bsy) Z_{\ell-1}}
    \phi(q^{(\ell-1)}(\bsy))
    \right)
    \bpi^\delta_\ell(\dd\bsy)
    \;,
\end{align}
see \eqref{eq:EpiL}.
This form corresponds to the splitting which is 
customary in MCMC and SMC methods, see, e.g.~\cite{BJLTZ15}
and the references there. 
We expect the application of QMC quadratures 
in finite arithmetic to the integrand function in \eqref{eq:DiffEx} to be
advantageous for small observation covariance $\Gamma$.
The integrand function of \eqref{eq:DiffEx} is, 
\emph{upon suitable ($\Gamma$-dependent) rescaling of coordinates as in \cite[Theorem 4.1]{SS14},}
$(\bsb,p,\eps)$-holomorphic \emph{uniformly with respect to $\Gamma>0$}.
\end{remark}
\subsection{Error analysis}
\label{sec:ErrAn}
Each of the proposed computable 
estimators \eqref{eq:QLratio}, \eqref{eq:QLsplit}
is based on approximating posterior expectations
of PG discretizations and dimensionally truncated forward models
by HoQMC integration.
The ensuing error analysis is based on the dimension independent
error bounds for these quadratures which we developed in 
\cite{DKGNS13,DKGS14}. 
These bounds are based on estimates of higher 
order derivatives of the integrand functions with respect to the integration
variables. 
In \cite[Section 3]{DLGCS14}, 
we developed estimates for these derivatives based on analytic continuation
of integrand functions into the complex domain and Cauchy's integral formula;
this is feasible for holomorphic forward problems considered in \cite{CCS2},
and for general, holomorphic potentials $\Phi_\Gamma$.
As is by now well-known, and in contrast to our 
single-level QMC error analysis for the quotient estimator in \cite{DGLGCSBIPSL},
multilevel QMC error estimates involve bounding
\emph{quadrature errors of PG forward problem discretization errors}.
To use the higher order QMC error bounds of \cite{DLGCS14} 
in the present multilevel context therefore requires 
\emph{
bounds on the dimension truncation and PG discretization 
errors of analytic continuations of the integrand functions.
}
These bounds are to hold uniformly 
with respect to the truncation dimension $s_\ell$.
The present, analytic continuation approach should be
contrasted with the ``real-variable'' approach used in
the error analysis of QMC integration in \cite{DKGS14}, 
which is based on bootstrapping arguments and induction.
%
\subsubsection{Holomorphy assumptions on the PG approximation}
\label{sec:HolPG}
To ensure holomorphy of countably parametric 
families of integrand functions in the Bayesian estimators
\eqref{eq:QLratio}, \eqref{eq:QLsplit}, 
as required by the QMC error bounds stated in Proposition \ref{prop:main1},
we impose corresponding holomorphy assumptions on the parametric
forward problem \eqref{eq:paraOpEq} as well as on its PG discretization
\eqref{eq:parmOpEqh_trun}. We formalize these conditions
(which are versions of the uniform parametric regularity 
{\bf UPR} and the parametric stability) in the following:

(i) ${\rm {\bf HCP}_t}$ Holomorphic continuation 
of the parametric forward problem in $\bcX_t$:
there is a $(\bsb,p,\eps)$-holomorphic extension of the parametric forward problem \eqref{eq:mainstrunc} 
such that, for any truncation dimension $s\in \IN$, and for every 
$\bsz\in O_\bsrho$, the extended problem
\be\label{eq:mainstrunC}
\mbox{find}\;q^s(\bsz_{\{ 1:s\}}) \in \bcX:\quad 
{ _{\bcY'} \langle \cR(u^s(\bsz);q^s(\bsz)), w \rangle_{\bcY} } = 0 
\;\;\forall w \in \bcY 
\;,
\ee
admits a unique, parametric solution $q^s(\bsz_{\{ 1:s\}}) \in \bcX_t$
which is  $(\bsb,p,\eps)$-holomorphic. 

(ii) {\bf PGStabC} Stability and Quasioptimality of the PG discretization 
for the complex parametric extension: 
for every $\bsz \in O_\bsrho$, and for every $0<h \leq h_0$
the Galerkin approximations: 
\begin{equation} \label{eq:parmOpEqh_bz}
\mbox{find} \; q_h(\bsz) \in \bcX^h :
\quad
{_{\bcY'}}\langle \cR(\bsz; q_h(\bsz)), w^h \rangle_{\bcY} = 0
\quad 
\forall w^h\in \bcY^h\;,
\end{equation}
are uniquely defined and converge quasioptimally:
there exists a constant $C>0$ such that 
for all $\bsz\in O_\bsrho$, 
with a $(\bsb_t, \varepsilon)$-admissible poly-radius $\bsrho$,
\begin{equation} \label{eq:quasiopt_bz}
 \| q(\bsz) - q_h(\bsz) \|_{\cX}
 \,\le\, 
\frac{C}{\bar{\mu}} \inf_{0\ne v^h\in \cX^h} \| q(\bsz) - v^h\|_{\cX}
\;.
\end{equation}
We remark that  ${\rm {\bf HCP}_t}$ and  {\bf PGStabC} imply, 
via the approximation property \eqref{eq:apprprop}, 
the convergence rate $O(h^t)$ in $\bcX$ for the 
Galerkin solution $q_h(\bsz) \in \bcX^h$, with implied 
constant independent of the truncation dimension $s$. 

(iii) 
${\rm {\bf ANPGC}}_{t'}$ Aubin-Nitsche argument for the 
PG discretization of the complex parametric extension:
There exists a constant $C>0$ 
such that for every $G\in \bcX'_{t'}$ there holds 
a superconvergence estimate: 
for every $\bsrho$ which is $(\bsb_t,\eps)$-admissible
for sufficiently small $\eps > 0$,
\be\label{eq:ANPGC}
\forall  \bsz \in \calO_\bsrho:\quad \left| G(q(\bsz)) - G(q_h(\bsz)) \right| 
\leq Ch^\tau\;, \quad \tau = t+t' 
\;.
\ee
Examples with valid  conditions (i) - (iii) will be presented
in the numerical experiments Section \ref{sec:NumExp} ahead;
the hypotheses will be verified in particular
for affine-parametric operator equations in Section \ref{sec:VerHyp}.
\subsubsection{Error bounds for the multilevel ratio estimator}
\label{sec:ErrAnRat}
Numerator and denominator of 
the estimator \eqref{eq:QLratio} are 
approximated separately by MLHoQMC estimators,
\eqref{eq:MLratioZ'}, \eqref{eq:MLratioZ}.
We first analyze the error of these approximations,
thereby also generalizing the single-level results in \cite{DLGCS14}.
\emph{
Throughout, we assume that in all densities which occur in the 
integrals \eqref{eq:Z} - \eqref{eq:intpsi} 
the parametric forward problems are dimension-truncated to 
finite parameter dimensions $\{s_\ell\}_{\ell = 0}^L$ 
for discretization levels $\ell = 0,1,...,L$, which 
we assume to be strictly increasing}
(the ensuing error analysis remains valid 
with obvious modifications, if
several or all truncation dimensions are equal)
i.e.
\[
0< s_0 < s_1 < ... < s_L <\infty \;.
\]
Throughout the error analysis, we assume $L\geq 2$ and $\theta \in \{1,2\}$
is as in \eqref{eq:Idimtrunc}. 

\begin{theorem}
Assume that $0 \leq t \leq \bar{t}$, for some $\bar{t} > 0$
and that, in addition, there exists $C_0 > 0$ such that
$|Z_\ell| > C_0$, $|Z_{\ell,\ell}| > C_0$ 
uniformly with respect to $\ell=0,1,2,\ldots$.
Assume also ${\rm {\bf HCP}_t}$, ${\bf PGStabC}$ 
hold for some $t>0$.

Then, for every $p_t < \lambda, q < 1$,
with $\theta \in  \{1,2\}$ as in \eqref{eq:Idimtrunc},
\begin{equation}
    \left\| \frac{Q^*_L[Z'_L] } {Q^*_L[Z_L]} - \bbE^{\bpi^\delta}[\phi] \right\|_{\cZ}
\le
C \left[ h_L^t + s_L^{-\theta(1/p_0-1)} + N_0^{-1/\lambda}
   + \sum_{\ell=1}^L N_\ell^{-1/\lambda} 
           ( h_{\ell-1}^{t} + s_{\ell-1}^{-\theta(1/p_0 - 1/q)} )
\right].
\end{equation}
\end{theorem}
\begin{proof}
Since $ \bbE^{\bpi^\delta}[\phi] = Z'/Z$, we estimate 
\be \label{eq:MLZ'Bd}
Z' - Q_L^\ast[Z_L'] = Z' - Z_L' + Z_L' - Q_L^\ast[Z_L'] =: I_L + II_L \;.
\ee
The first term $I_L$ is the exact posterior expectation on the
approximate forward model, i.e.~a pure discretization error. 
It can be estimated using Proposition \ref{prop:ErrBayesExpec} 
and the approximation property \eqref{eq:apprprop} 
with property ${\bf HCP}_t$.
There exists a constant $C>0$ 
which is independent of $h_L$ and of $s_L$ such that
\be\label{eq:DiscErr} 
\| Z' - Z_L' \|_{\cZ} \leq C (h_L^t + s_L^{-\theta(1/p_0-1)}) 
\;.
\ee
%
To bound term $II_L$ in \eqref{eq:MLZ'Bd}, we write it as telescoping sum
\be\label{eq:TlcpID}
Z_L' - Q_L^\ast[Z_L'] 
= 
(I-Q_{N_0,s_0})[\phi(q^0) \Theta_0 ] 
+
\sum_{\ell=1}^L
\left[
(I-Q_{N_\ell,s_\ell})[\phi(q^{(\ell)}) 
\Theta_{\ell} - \phi(q^{(\ell-1)}) \Theta_{\ell-1} ]
\right]
\;.
\ee
%
%
The first term in \eqref{eq:TlcpID} is a pure QMC integration error,
on the coarsest level PG discretization; it can be estimated
with \eqref{qmc_upper_bound}. The remaining terms are QMC 
quadrature errors of dimension truncation and Petrov-Galerkin errors.
A typical term reads 
\begin{equation}\label{eq:QMCPGTT}
(I-Q_{N_\ell,s_\ell})[\phi( q^{(\ell)} ) \Theta_{\ell} - \phi(q^{(\ell-1)} ) \Theta_{\ell-1} ]\;,
\quad \ell = 1,..., L \;.
\end{equation}
We rewrite a generic term in the telescoping sum \eqref{eq:TlcpID} as
(parameter $\bsz$ not indicated)
\[
\phi(q^{(\ell)}) \Theta_{\ell} - \phi(q^{\ell-1}) \Theta_{\ell-1}
=
\phi(q^{(\ell) - q^{(\ell-1)} }) \Theta_{\ell}
+
\phi(q^{(\ell-1)} )( \Theta_{\ell} - \Theta_{\ell-1} ) 
=: 
D^\ell_\phi \Theta_{\ell} + \phi(q^{(\ell-1)} ) D^\ell_\Theta
\;.
\]
Here, we used the linearity of the QoI $\phi( )$.

To apply the QMC quadrature error bound \eqref{qmc_upper_bound}
with 
$g(\bsz) = (\phi(q^{(\ell) } ) \Theta_{\ell} - \phi( q^{(\ell-1)}) \Theta_{\ell-1} )(\bsz)$,
we need to estimate (cp. \eqref{eq:TrhoNZ}) for 
$\bsz \in \cT_\bsrho$ with $\bsrho$ being $(\bsbeta_t,\eps)$-admissible 
%
\be\label{eq:DetEst}
\begin{array}{rcl}
\| g \|_{\bsbeta,\cZ} 
&:=& \displaystyle  
\sup_{\bsz \in \cT_\bsrho} 
\| (\phi( q^{(\ell)}) \Theta_{\ell} - \phi( q^{(\ell-1)} ) \Theta_{\ell-1})(\bsz) \|_\cZ 
\\
&\leq& \displaystyle
\sup_{\bsz \in \cT_\bsrho} 
\| D^\ell_\phi(\bsz) \Theta_{\ell}(\bsz) \|_\cZ 
+ 
\sup_{\bsz \in \cT_\bsrho}
\| \phi( q^{(\ell-1)} )(\bsz) D^\ell_\Theta(\bsz) \|_\cZ
\;.
\end{array}
\ee
We estimate each of the two differences in essentially the same way. 
Consider $D^\ell_\phi$. Then
\[
\begin{array}{rcl} 
\| D^\ell_\phi(\bsz) \|_\cZ & \leq & \displaystyle 
\| \phi \|_{\cL(\bcX,\cZ)} \| q^{(\ell)} (\bsz) - q^{(\ell-1)} (\bsz) \|_{\bcX} 
\\
& = & \displaystyle 
\| \phi \|_{\cL(\bcX,\cZ)} 
\| q^{s_\ell}_{h_\ell}(\bsz) - q^{s_\ell}_{h_{\ell-1}}(\bsz) 
+  q^{s_\ell}_{h_{\ell-1}}(\bsz) - q^{s_{\ell-1}}_{h_{\ell-1}}(\bsz) \|_{\bcX} 
\\
& \leq  & \displaystyle 
\| \phi \|_{\cL(\bcX,\cZ)} 
\left[
\| q^{s_\ell}(\bsz) - q^{s_\ell}_{h_\ell}(\bsz) \|_{\bcX}
+
\| q^{s_\ell}(\bsz) - q^{s_\ell}_{h_{\ell-1}}(\bsz) \|_{\bcX}
+
\| q^{s_\ell}_{h_{\ell-1}}(\bsz) - q^{s_{\ell-1}}_{h_{\ell-1}}(\bsz) \|_{\bcX}
\right]
\;.
\end{array}
\]
The first two terms are PG discretization errors 
for the complex-parametric extension of the
operator equation. 
Property ${\rm {\bf HCP}_t}$ 
(holomorphic continuability of the parametric forward problem)
implies the well-posedness of the complex-parametric problem, 
and assumption {\bf PGStabC}
(stability of the PG discretization in the complex-parametric problem) 
implies that the PG discretization errors can be bounded as
\be\label{eq:PGerr}
\sup_{\bsz \in \cT_{\bsrho}} 
\left[ 
\| q^{s_\ell}(\bsz) - q^{s_\ell}_{h_\ell}(\bsz) \|_{\bcX}
+
\| q^{s_\ell}(\bsz) - q^{s_\ell}_{h_{\ell-1}}(\bsz) \|_{\bcX}
\right]
\leq 
C h_{\ell-1}^t \sup_{\bsz \in \cT_{\bsrho}} \| q(\bsz) \|_{\bcX_t} \;,
\ee
provided that $\bsrho$ is $(\bsbeta_t,\eps)$-admissible.
The third term is a dimension truncation error. 
It vanishes if $s_{\ell-1} = s_{\ell}$, 
and by the stability of the PG discretization 
is bounded for any $p_t \leq q \leq 1$ by
\be\label{eq:DTerr}
\sup_{\bsz \in \cT_{\bsrho}}
\| q^{s_\ell}_{h_{\ell-1}}(\bsz) - q^{s_{\ell-1}}_{h_{\ell-1}}(\bsz) \|_{\bcX}
\leq 
C s_{\ell-1}^{-\theta(1/p_0 - 1/q)} 
\;,
\;\; \ell =1,2,...,L-1
\;.
\ee
%
%
As the modulus of the posterior density, $| \Theta_{\ell} (\bsz)|$, 
is uniformly bounded w.r.t. $\ell$ and w.r.t. $\bsz \in \cT_\bsrho$, 
for any $(\bsbeta_0,\eps)$-admissible $\bsrho$, 
the first term in the bound \eqref{eq:DetEst} 
can be estimated for $L\geq 2$ by an absolute multiple of
\be\label{eq:CmbBd}
 h_{\ell-1}^t + s_{\ell-1}^{-\theta(1/p_0 - 1/q)} \;, \quad \ell = 1,2,...,L-1 \;.
\ee
For a fixed $(\bsbeta,\eps)$-admissible poly-radius $\bsrho$, and 
for any $\bsz\in \cT_\bsrho$ we estimate the second term in \eqref{eq:DetEst}
by
\be\label{eq:2ndTerm}
\sup_{\bsz \in \cT_\bsrho}
\| \phi(q^{(\ell-1)})(\bsz) D^\ell_\Theta(\bsz) \|_\cZ
= 
\sup_{\bsz \in \cT_\bsrho} | D^\ell_\Theta(\bsz)| 
\sup_{\bsz \in \cT_\bsrho} \| \phi( q^{(\ell-1)} )(\bsz) \|_\cZ
\;.
\ee
Here, the latter supremum is bounded uniformly w.r.t. $\ell$, due to 
condition {\bf PGStabC}:
the uniform stability of the PG discretization for the complex-parametric problem
implies also the existence of a bound $B(\bsbeta,\eps)$,
independent of the discretization level $\ell$ such that
\[
\forall \bsrho \mbox{ which are $(\bsbeta,\eps)$-admissible}: \;\;
\sup_{\bsz \in \cT_\bsrho} \| q^{(\ell)}(\bsz) \|_\bcX \leq B(\bsbeta,\eps) <\infty \;.
\]
We focus on the first supremum in \eqref{eq:2ndTerm}.
\be \label{eq:PstDnsEst}
\begin{array}{rcl}
\sup_{\bsz \in \cT_\bsrho} | D^\ell_\Theta(\bsz)|
& = & \displaystyle 
\sup_{\bsz \in \cT_\bsrho} | \Theta^{s_\ell}_{h_\ell}(\bsz) -  \Theta^{s_\ell}_{h_{\ell-1}}(\bsz) 
                           +\Theta^{s_\ell}_{h_{\ell-1}}(\bsz) - \Theta^{s_{\ell-1}}_{h_{\ell-1}}(\bsz)|
\\
& \leq & \displaystyle 
\sup_{\bsz \in \cT_\bsrho} 
\left[
| \Theta^{s_\ell}(\bsz) - \Theta^{s_\ell}_{h_\ell}(\bsz)|
+ 
|\Theta^{s_\ell}(\bsz) -  \Theta^{s_\ell}_{h_{\ell-1}}(\bsz)|
+
|\Theta^{s_\ell}_{h_{\ell-1}}(\bsz) - \Theta^{s_{\ell-1}}_{h_{\ell-1}}(\bsz)|
\right]
\;.
\end{array}
\ee
We deal with three errors in the posterior density, 
which arise due to different approximations in the forward maps. 

To reduce the error in the posterior density to an
error bound in the corresponding potentials $\Phi_\Gamma$,
we recall definition \eqref{eq:DefPhiu} (applied with transposition to
the complex-parametric forward map) and write
\[
\left|
\int_{-\Phi_\Gamma}^{-\Phi^M_\Gamma} e^\zeta d\zeta 
\right|
=
\left| \exp(-\Phi^M_\Gamma) - \exp(-\Phi_\Gamma) \right|
\leq 
| \Phi^M_\Gamma - \Phi_\Gamma | 
\max_{\zeta \in {\rm conv}(-\Phi^M_\Gamma, -\Phi_\Gamma)} |e^\zeta|
\],
where the complex integral is along a straight line in $\IC$ connecting
$-\Phi^M_\Gamma$ and $-\Phi_\Gamma$.

The difference $ \Phi_\Gamma - \Phi^{h_\ell}_\Gamma $ 
in the Bayesian potentials 
due to the PG discretization (at equal truncation dimension $s_\ell$) 
of the forward problem
is estimated using (a complex-parametric extension of) \eqref{eq:qConsis} 
again in terms of the PG forward discretization error. 

We obtain for every $\bsrho$ which is $(\bsbeta_t,\eps)$-admissible
the bound
\[
\sup_{\bsz \in \cT_\bsrho} 
\left[
| \Theta^{s_\ell}(\bsz) - \Theta^{s_\ell}_{h_\ell}(\bsz)|
+ 
|\Theta^{s_\ell}(\bsz) -  \Theta^{s_\ell}_{h_{\ell-1}}(\bsz)|
\right]
\leq 
C(t,\eps) h_{\ell-1}^t
\;.
\]
The third term in \eqref{eq:PstDnsEst} is bounded in the same way 
in terms of the error \eqref{eq:DTerr}, resulting also for \eqref{eq:2ndTerm}
in the upper bound \eqref{eq:CmbBd}.
Collecting all bounds, we have shown that
for any poly-radius $\bsrho$ which is $(\bsbeta_t,\eps)$-admissible,
there exists a constant $C(\bsbeta_t,\eps) > 0$ such that 
for any $p_t \leq \lambda < 1$ and for all $\bsz \in  \cT_\bsrho$ 
\[
\begin{array}{rcl}
\| Z_L' - Q_L^\ast[Z_L']\|_\cZ 
& \leq  & \displaystyle 
\| (I-Q_{N_0,s_0})[\phi(q^0) \Theta_0 ] \|_\cZ
+
\sum_{\ell=1}^L
\left\|
(I-Q_{N_\ell,s_\ell})[\phi(q^{(\ell)}) \Theta_{\ell} - \phi(q^{(\ell-1)}) \Theta_{\ell-1}]
\right\|_\cZ
\\
& \leq  & \displaystyle
C\left[ N_0^{-1/p_0} 
+ \sum_{\ell=1}^L N_\ell^{-1/\lambda} ( h_{\ell-1}^t + s_{\ell-1}^{-\theta(1/p_0 - 1/q)} )
\right]
\;.
\end{array}
\]
Combining this with \eqref{eq:DiscErr}, we obtain from \eqref{eq:MLZ'Bd} the error estimate
\be\label{eq:ZZ'Est}
\| Z' - Q_L^\ast[Z_L'] \|_\cZ 
\leq 
C 
\left[ h_L^t + s_L^{-\theta(1/p_0-1)} + N_0^{-1/p_0} 
+ 
\sum_{\ell=1}^L N_\ell^{-1/\lambda} ( h_{\ell-1}^t + s_{\ell-1}^{-\theta(1/p_0 - 1/q)} )
\right]
\;.
\ee
By choosing $\phi \equiv 1$, and $\cZ = \IR$, the bound \eqref{eq:ZZ'Est}
applies also for the approximation $Q_L^\ast[Z_L]$ of $Z$. 

We may now estimate the error in the quotient estimator \eqref{eq:QLratio}.
To this end, we write
\[
\begin{array}{rcl}
\displaystyle
\left\| \frac{Z'}{Z} - \frac{Q^*_L[Z']}{Q^*_L[Z]} \right\|_\cZ
& = & \displaystyle 
\left\| \frac{Z'Q^*_L[Z] - Q^*_L[Z'] Z}{ZQ^*_L[Z]} \right\|_\cZ
\\
& \leq & \displaystyle
\frac{1}{ZQ^*_L[Z]}
\left\{ 
\| Z' \|_{\cZ} | Z-Q^*_L[Z] | + |Z| \| Z' - Q^*_L[Z'] \|_\cZ 
\right\}\;.
\end{array}
\]
Using \eqref{eq:ZZ'Est} twice, we arrive at the error bound for the
ratio estimator,
\be\label{eq:RatEstBd}
\left\| \frac{Z'}{Z} - \frac{Q^*_L[Z']}{Q^*_L[Z]} \right\|_\cZ
\leq 
C \left[ h_L^t + s_L^{-\theta(1/p_0-1)} + N_0^{-1/p_0}
   + \sum_{\ell=1}^L N_\ell^{-1/\lambda} 
                  ( h_{\ell-1}^t + s_{\ell-1}^{-\theta(1/p_0 - 1/q)} )
\right]
\;.
\ee
\end{proof}
%
\subsubsection{Error bounds for the multilevel splitting estimator}
\label{sec:ErrAnSplit}
\begin{theorem}
Assume that $0 \leq t \leq \bar{t}$, for some $\bar{t} > 0$
and that, in addition, there exists $C_0 > 0$ such that
$|Z_\ell| > C_0$, $|Z_{\ell,\ell}| > C_0$ 
uniformly with respect to $\ell=0,1,2,\ldots$.
Assume ${\rm {\bf HCP}_t}$, ${\bf PGStabC}$
hold for some $t>0$.
Then, for every $p_t < \lambda, q < 1$,
with $\theta \in  \{1,2\}$ as in \eqref{eq:Idimtrunc},
\begin{equation} \label{eq:SpltEstBd}
    \big\| Q_{L,split}^{*}[\phi] - \bbE^{\bpi^\delta}[\phi] \big\|_{\cZ}
\le
C \left[ h_L^t + s_L^{-\theta(1/p_0-1)} + N_0^{-1/\lambda}
   + \sum_{\ell=1}^L N_\ell^{-1/\lambda} 
           ( h_{\ell-1}^t + s_{\ell-1}^{-\theta(1/p_0 - 1/q)} )
\right].
\end{equation}
\end{theorem}
\begin{proof}

For given sequences 
$\{ s_\ell \}_{\ell = 0}^L$, $\{ N_\ell \}_{\ell=0}^L$ 
of truncation dimensions $s_\ell$ and numbers $N_\ell$ of QMC points, 
and for $\mathbb{E}^{\pi^\delta_\ell}$ given by \eqref{eq:EpiL}, 
we estimate
\begin{equation*}
    \big\| \bbE^{\bpi^\delta}[\phi] - Q_L^{*}[\phi] \big\|_{\cZ}
    \le
    \big\| \bbE^{\bpi^\delta} [\phi] - \bbE^{\bpi^\delta_L}[\phi] \big\|_{\cZ} 
     +
    \big\| \bbE^{\bpi^\delta_L} [\phi]  - Q_L^{*}[\phi] \big\|_{\cZ}
\;.
\end{equation*}
To bound the first term, we use 
Proposition~\ref{prop:ErrBayesExpec} (see \eqref{eq:BayAppr})
and the approximation property \eqref{eq:apprprop} 
and the dimension truncation estimate \eqref{eq:Vdimtrunc}
with the discretization parameters $\{ (s_\ell, N_\ell) \}_{\ell = 0,...,L}$,
\begin{equation*}
    \big\| \bbE^{\bpi^\delta}[\phi] - \bbE^{\bpi^\delta_L}[\phi] \big\|_{\cZ}
    \le
    C(\Gamma,R,r) \left( h^t_L + s_L^{-\theta(1/p_0-1)} \right) 
\;.
\end{equation*}
To bound the second term, we write, after 
subtracting \eqref{eq:BIsplit} and \eqref{eq:QLsplit}
\[
\bbE^{\pi^\delta_L}[\phi] - Q^{*}_{L,split}
= \frac{Z'_0}{Z_0} - \frac{Z'_{0,0}}{Z_{0,0}}
  + \sum_{\ell=1}^L 
\left[\left( \frac{Z'_{\ell}}{Z_{\ell}}
 - \frac{Z'_{\ell,\ell}}{Z_{\ell,\ell}} \right)
    - \left( \frac{Z'_{\ell-1}}{Z_{\ell-1}}
  - \frac{Z'_{\ell-1,\ell}}
   {Z_{\ell-1,\ell}} \right) \right]
\;.
\]
We rewrite the difference between the $\ell$-th term in 
the sum over $\ell = 1,...,L$ in \eqref{eq:BIsplit} and \eqref{eq:QLsplit}
as
\begin{align*}
\left( \frac{Z'_{\ell}}{Z_{\ell}}
 - \frac{Z'_{\ell,\ell}}{Z_{\ell,\ell}} \right)
    - \left( \frac{Z'_{\ell-1}}{Z_{\ell-1}}
  - \frac{Z'_{\ell-1,\ell}}
   {Z_{\ell-1,\ell}} \right)
 &=
\underbrace{\frac{1}{Z_\ell} (Z'_{\ell} - Z'_{\ell,\ell})
- \frac{1}{Z_{\ell-1}} (Z'_{\ell-1} - Z'_{\ell-1,\ell})}_{A_\ell} \\
 & + 
\underbrace{Z'_{\ell,\ell}
  \left( \frac{1}{Z_\ell} - \frac{1}{Z_{\ell,\ell}}\right) 
-
 Z'_{\ell-1,\ell}
\left( \frac{1}{Z_{\ell-1}} - \frac{1}{Z_{\ell-1,\ell}} \right)}_{B_\ell}
\;.
\end{align*}
We rewrite $A_\ell$ upon adding and subtracting the term 
$(Z'_{\ell-1}-Z'_{\ell-1,\ell})/Z_{\ell} $ as
\begin{align}
A_{\ell} &= A_{\ell} + 
 (Z'_{\ell-1}-Z'_{\ell-1,\ell})/Z_{\ell}
- (Z'_{\ell-1}-Z'_{\ell-1,\ell})/Z_{\ell}\nonumber \\
&=
\frac{1}{Z_\ell}[ (Z'_{\ell}-Z'_{\ell,\ell}) - (Z'_{\ell-1} - Z'_{\ell-1,\ell})]
+ (Z'_{\ell-1}- Z'_{\ell-1,\ell}) 
\left( \frac{1}{Z_\ell} - \frac{1}{Z_{\ell-1}} \right)
\label{eq:Aell}
\;.
\end{align}
For the first term in the right hand side of \eqref{eq:Aell},
we observe that
\begin{align*}
(Z'_{\ell} - Z'_{\ell,\ell}) - (Z'_{\ell-1} - Z'_{\ell-1,\ell})
&= (I[\phi(q^{(\ell)}) \Theta_{\ell }] 
  - Q_{N_\ell}[\phi(q^{(\ell)}) \Theta_{\ell} ])
- (I[\phi( q^{(\ell-1) } ) \Theta_{\ell-1}] 
  - Q_{N_\ell}[\phi( q^{(\ell-1)}) \Theta_{\ell-1} ]) 
\\
&= 
(I-Q_{N_{\ell}}) 
[ \phi(q^{(\ell)}) \Theta_{\ell} - \phi(q^{(\ell-1)}) \Theta_{\ell-1} ]
\;.
\end{align*}
Using the estimates of \eqref{eq:QMCPGTT}
in the previous section we obtain 
\[
\left\|
(Z'_{\ell}-Z'_{\ell,\ell}) - (Z'_{\ell-1} - Z'_{\ell-1,\ell})
\right\|_\cZ
\le 
C N_\ell^{-1/\lambda} ( h^t_{\ell-1} + s_{\ell-1}^{-\theta(1/p_0 - 1/q)})
\]
where $C>0$ is independent of $\{s_{\ell}\}_{\ell\geq 0}$ and of $N_\ell$.
For the second term in \eqref{eq:Aell}, we have
\[
\|Z'_{\ell-1} - Z'_{\ell-1,\ell}\|_\cZ \le CN_\ell^{-1/\lambda}
\]
and
\[
\left| \frac{1}{Z_\ell} - \frac{1}{Z_{\ell-1}} \right|
= 
\frac{|Z_{\ell-1}- Z_{\ell}|}{|Z_\ell Z_{\ell-1}|}
\le 
C (h^t_{\ell-1} + s_{\ell-1}^{-\theta(1/p_0-1/q)})
\;.
\]
Adding and subtracting the term
$Z'_{\ell-1,\ell}(1/Z_{\ell}-1/Z_{\ell,\ell})$,
we may rewrite $B_\ell$ as
\begin{align}
B_\ell &= B_\ell 
- Z'_{\ell-1,\ell}\left( \frac{1}{Z_{\ell-1}} - \frac{1}{Z_{\ell,\ell}} \right)
+ Z'_{\ell-1,\ell}\left( \frac{1}{Z_{\ell-1}} - \frac{1}{Z_{\ell,\ell}} \right)\nonumber 
\\
&= (Z'_{\ell,\ell}- Z'_{\ell-1,\ell})
\left( \frac{1}{Z_\ell} - \frac{1}{Z_{\ell,\ell}} \right)
+ Z'_{\ell-1,\ell}
\left( \frac{1}{Z_\ell} - \frac{1}{Z_{\ell,\ell}} 
 - (\frac{1}{Z_{\ell-1}} - \frac{1}{Z_{\ell-1,\ell}}) \right)
\;.
\label{eq:last}
\end{align}
Using Proposition~\ref{prop:main1} and the fact that
$| \Theta_{\ell} (\bz)|$ is uniformly bounded with
respect to $\ell$ and $\bz \in \cT_\rho$, we have
\begin{equation}\label{eq:Zellell}
| Z_{\ell,\ell} - Z_{\ell} | 
=
| Q_{N_{\ell}}[\Theta_{\ell}] - I[\Theta_{\ell}] |
\le 
CN_{\ell}^{-1/\lambda}
\;.
\end{equation}
If $|Z_\ell| > C_0$, $|Z_{\ell,\ell}| > C_0$ 
uniformly with respect to $\ell=0,1,2,\ldots$,
for some $C_0 > 0$, then there exists $C>0$ 
(depending on $C_0$, but independent of $\ell$ or $N_\ell$)
such that
\begin{equation}\label{eq:overZell}
\left| \frac{1}{Z_\ell} - \frac{1}{Z_{\ell,\ell}} \right|
\le 
\frac{|Z_{\ell,\ell} - Z_{\ell}|} {|Z_\ell| |Z_{\ell,\ell}| }
\le 
C N_\ell^{-1/\lambda}
\;.
\end{equation}
Using similar techniques as in the estimates of \eqref{eq:PstDnsEst}, 
\[
\|Z'_{\ell,\ell} - Z'_{\ell-1,\ell}\|_{\cZ} 
\le 
C( h^t_{\ell-1} + s_{\ell-1}^{-\theta(1/p_0-1/q)})
\;.
\]
Thus, for any sequence  $\{h_\ell\}_{\ell \geq 0}$ or $\{ s_\ell \}_{\ell \geq 0}$, 
the $\| \cdot \|_\cZ$-norm of the first term of \eqref{eq:last} is 
bounded by 
$O(N_\ell^{-1/\lambda} (h^t_{\ell-1} + s_{\ell-1}^{-\theta(1/p_0-1/q)}))\;$,
%
for any $p_t\leq\lambda,q < 1$.

For the second term of \eqref{eq:last},
adding and subtracting
the same term $(Z_{\ell-1,\ell}-Z_{\ell-1})/(Z_{\ell}Z_{\ell,\ell})$
\begin{align}
\frac{1}{Z_\ell} - \frac{1}{Z_{\ell,\ell}} 
 - (\frac{1}{Z_{\ell-1}} - \frac{1}{Z_{\ell-1,\ell}})
&= \frac{Z_{\ell,\ell}-Z_{\ell}} {Z_{\ell} Z_{\ell,\ell}}
 - \frac{Z_{\ell-1,\ell}-Z_{\ell-1}}{ Z_{\ell} Z_{\ell,\ell}} 
 + \frac{Z_{\ell-1,\ell}-Z_{\ell-1}}{ Z_{\ell} Z_{\ell,\ell}} 
 - \frac{Z_{\ell-1,\ell} - Z_{\ell-1}}{Z_{\ell-1}Z_{\ell-1,\ell}}
\nonumber \\
& =\frac{1}{Z_{\ell} Z_{\ell,\ell}}
 [(Z_{\ell,\ell}-Z_\ell)-(Z_{\ell-1,\ell}-Z_{\ell-1})]
\label{eq:Cell} 
\\
 & \quad + \left( \frac{1}{Z_{\ell} Z_{\ell,\ell}} - 
  \frac{1}{Z_{\ell-1} Z_{\ell-1,\ell}} \right)(Z_{\ell-1,\ell}-Z_{\ell-1})
\;.
\label{eq:Dell}
\end{align}
We bound \eqref{eq:Cell} by similar techniques 
as in the estimates \eqref{eq:QMCPGTT} to arrive at the bound
\[
C N_\ell^{-1/\lambda}(h^t_{\ell-1} + s^{-\theta(1/p_0-1/q)}_{\ell-1})
\;.
\]
For the remaining term \eqref{eq:Dell}, we rewrite the first factor as
\begin{equation}\label{eq:4_45a}
\frac{1}{Z_{\ell} Z_{\ell,\ell}}
- \frac{1}{Z_{\ell-1} Z_{\ell-1,\ell}}
= \frac{Z_{\ell-1}Z_{\ell-1,\ell} - Z_{\ell} Z_{\ell,\ell}}
{ Z_{\ell} Z_{\ell,\ell} Z_{\ell-1} Z_{\ell-1,\ell}}
= \frac{Z_{\ell-1,\ell}(Z_{\ell-1} - Z_{\ell})
   +Z_\ell(Z_{\ell-1,\ell}- Z_{\ell,\ell})}
{ Z_{\ell} Z_{\ell,\ell} Z_{\ell-1} Z_{\ell-1,\ell}}
\;.
\end{equation}
Using similar estimates as in \eqref{eq:PstDnsEst}, 
we conclude
that there exists $C>0$ independent of 
$\{s_\ell\}_{\ell \geq 0}$ and of $\{h_\ell\}_{\ell \geq 0}$
such that 
\begin{equation} \label{eq:bounds4_45}
\begin{array}{rl} 
|Z_{\ell-1}-Z_{\ell}| 
&= |I( \Theta_{\ell-1} - \Theta_{\ell} )| 
\le C (h^t_{\ell-1} + s^{-\theta(1/p_0-1/q)}_{\ell-1}), 
\\
|Z_{\ell-1,\ell}-Z_{\ell,\ell}|
&= |Q_{N_{\ell}} ( \Theta_{\ell-1} - \Theta_{\ell} )|
\le 
\sup_{\bz \in \cT_\rho} | \Theta_{\ell-1}(\bz) - \Theta_{\ell}(\bz)|
\le 
C (h^t_{\ell-1}+s^{-\theta(1/p_0-1/q)}_{\ell-1})
\;.
\end{array}
\end{equation}
All $|Z_{\ell}|$ are bounded above and below away from zero
for $\ell=0,1,2,\ldots$, and the first factor of \eqref{eq:Dell}
is bounded by the RHS of \eqref{eq:bounds4_45}. 
The second factor of \eqref{eq:Dell} is estimated by
\[
  |Z_{\ell-1,\ell} - Z_{\ell-1}| = 
  |Q_{N_{\ell}}[ \Theta_{\ell-1}] - I [\Theta_{\ell-1}]| 
  \le C N_{\ell}^{-1/\lambda}\;.
\]
Combining all estimates for \eqref{eq:Cell} and \eqref{eq:Dell}, 
we conclude that there exists a constant $C>0$ such that,
for every $p_t \leq \lambda, q \leq 1$, and for all
sequences  $\{h_\ell\}_{\ell \geq 0}$,
$\{ s_\ell \}_{\ell \geq 0}$ and $\{N_\ell\}_{\ell \geq 0}$ 
holds
%
\begin{align*}
\left|
\left( \frac{1}{Z_\ell} - \frac{1}{Z_{\ell,\ell}} \right)
 - \left(\frac{1}{Z_{\ell-1}} - \frac{1}{Z_{\ell-1,\ell}}\right)
\right|
&\le C N_\ell^{-1/\lambda} (h^t_{\ell-1} + s_{\ell-1}^{-\theta(1/p_0-1/q)})
\;.
\end{align*}
Therefore, both $A_\ell$ and $B_\ell$ in \eqref{eq:Aell}, \eqref{eq:last}
are bounded by
$C(N_\ell^{-1/\lambda}(h^t_{\ell-1} + s^{-(1/p_0-1/q)}_{\ell-1}))$
(with constant $C>0$ independent of $\{h_\ell\}_{\ell \geq 0}$, 
$\{ s_\ell \}_{\ell \geq 0}$ or $\{N_\ell\}_{\ell \geq 0}$).
Summing the preceding estimates over all discretization levels,
we arrive at the error bound \eqref{eq:SpltEstBd} for the splitting estimator.
\end{proof}
\subsubsection{Selection of truncation dimensions $\{s_\ell \}_{\ell \geq 0}$ and 
               sample numbers $\{ N_\ell \}_{\ell \geq 0}$}
\label{sec:QMCNell}
We assume in the following that assumption ${\rm {\bf ANPGC}}_{t'}$ holds for some $t'\geq 0$, 
and define $\tau=t+t'$.
The error analysis of the computable MLHoQMC Bayesian estimators 
in Sections \ref{sec:ErrAnRat} and \ref{sec:ErrAnSplit} lead to error bounds 
\eqref{eq:RatEstBd} and \eqref{eq:SpltEstBd}, respectively, which 
are of the generic form 
\begin{equation}\label{eq:ErrBdTau}
    h_L^\tau + s_L^{-\theta(1/p_0-1)} + N_0^{-1/p_0}
        + \sum_{\ell=1}^L
            N_\ell^{-1/\lambda} 
            ( h_{\ell-1}^\tau + s_{\ell-1}^{-\theta(1/p_0 - 1/q)} ),
\end{equation}
with $t'=0$, and which is analogous to the bounds proved for MLHoQMC algorithms
for forward UQ in \cite{DKGS14} under ${\rm {\bf ANPGC}}_{t'}$ with some $t' > 0$.
In \eqref{eq:ErrBdTau}, the maximal rate in $N_\ell$ is obtained for the choice
$\lambda=p_t$, implying $\alpha=\lfloor 1/p_t \rfloor + 1$ for the digit interlacing factor.
Moreover, to obtain the maximal rate for $s_{\ell-1}^{-\theta(1/p_0-1/q)}$ 
we choose $q$ as small as possible, which yields $q=p_t$ since $p_t\le q<1$.

\paragraph{Choice of truncation dimension.}
In a first step, we balance the dimension truncation error on level $L$
with the discretization error on level $L$,
and similarly for the increment levels $\ell=1,\ldots,L$.
We assume in the following the behavior
$h_\ell = 2^{-(\ell+\ell_0)}$ for $\ell=0,\ldots,L$ and $\ell_0\in\bbN_0$.
For the terms on level $L$, we have $s_L^{-\theta(1/p_0-1)} = \calO(h_L^\tau)$,
which motivates the choice $s_L = 2^{\frac{p_0 \tau (L+\ell_0)}{\theta(1-p_0)}}$.
Using $\frac{1}{p_t} = \frac{1}{p_0}-\frac{t}{d}$ 
(cp. \cite[Eqn. (1.12)]{DKGS14}), we have for the increments
$s_\ell^{-\theta(1/p_0-1/p_t)} = s_\ell^{-\theta t/d} = \calO(h_\ell^\tau)$,
which leads to $s_\ell = 2^{\tau d (\ell+\ell_0)/(\theta t)}$.
Since the discretization error on mesh level $L$ 
limits the accuracy of the entire computation,
increasing the parameter space dimension beyond $s_L$ is unnecessary;
we therefore choose the truncation dimensions 
\begin{equation}\label{eq:sl}
    s_\ell = \big\lceil\min(
        2^{\tau d (\ell+\ell_0)/(\theta t)},
        2^{\frac{p_0 \tau (L+\ell_0)}{\theta(1-p_0)}}
    )\big\rceil\;.
\end{equation}

\paragraph{Error and work models.}
We now use $h_{\ell-1}/h_\ell = 2$ (any sequence $\{ h_\ell \}_{\ell \geq 0}$
with uniformly bounded ratios would be admissible in the error analysis)
and the preceding equilibration of the dimension truncation
and FEM discretization errors to obtain an expression for the total error.
We replace $N_0^{-1/p_0}$ by $N_0^{-1/p_t}$, 
since for $p_0\le p_t$ holds $N_0^{-1/p_0} \le N_0^{-1/p_t}$ for $N_0\in\bbN$,
\begin{equation}\label{eq:Etot}
    E_{tot}
    =
    \calO\left( h_L^\tau + \sum_{\ell=0}^L N_\ell^{-1/p_t} h_\ell^\tau \right)
    .
\end{equation}
The total work (cost) is given by
\begin{equation}\label{eq:Wtot}
    W_{tot} = \cO\left( \sum_{\ell=0}^L N_\ell h_\ell^{-d} s_\ell \right)
    .
\end{equation}

\paragraph{Optimization of the number of samples.}
We now seek to minimize the error bound \eqref{eq:Etot} 
subject to a given, fixed work budget for \eqref{eq:Wtot}.
A minimum can be found using a Lagrange multiplier $\Lambda$.
To ease the ensuing calculations,
we assume all constants in the asymptotic error bounds equal $1$. 
Then, we consider the Lagrangian 
\begin{equation*}
    \cL(\{ N_\ell \}_{\ell = 0}^L, \Lambda)
    :=
    h_L^\tau + \sum_{\ell=0}^L N_\ell^{-1/p_t} h_\ell^\tau
    + \Lambda \sum_{\ell=0}^L N_\ell h_\ell^{-d} s_\ell
    \;.
\end{equation*}
Now, we consider the necessary condition that $\partial\cL/\partial N_\ell = 0$
for all $\ell=0,\ldots,L$.
For now, we consider $N_0$ to be a free parameter (to be determined),
and use this to find an expression for the multiplier $\Lambda$ 
from the following condition:
\begin{equation}\label{eq:Lmultiplier}
    \frac{\partial\cL}{\partial N_0}
    = - \frac{N_0^{-1/p_t-1}}{p_t}h_0^\tau + \Lambda h_0^{-d}s_0 = 0
    \quad\Rightarrow\quad
    \Lambda = \frac{h_0^{\tau+d} N_0^{-(1+p_t)/p_t}}{p_ts_0}
    .
\end{equation}
Note that the assumption $h_0=s_0=1$
can lead to a large increase in the sample numbers in the MC case,
if it is not fulfilled.
In the numerical experiments ahead,
this assumption does not hold,
since we use $h_\ell=2^{-(\ell+\ell_0)}$ with $\ell_0=1$.
  
Inserting \eqref{eq:Lmultiplier} into the remaining conditions 
for $\ell=1,\ldots,L$ yields
\begin{equation*}
    N_\ell \sim
    N_0
    \left(
        \frac{h_\ell^{\tau+d}s_0}{h_0^{\tau+d}s_\ell}
    \right)^{p_t/(1+p_t)}
    ,\quad \ell=1,\ldots,L
    \,.
\end{equation*}
Inserting this expression for $N_\ell$ into \eqref{eq:Etot},
we set $C=(h_0^{\tau+d}s_0^{-1})^{1/(1+p_t)}$ and write the total error as
\begin{equation*}
    E_{tot} =
    \cO( h_L^\tau + C N_0^{-1/p_t} E )
    ,\quad \;\mbox{with}\;\;
    E := \sum_{\ell=0}^L \big(s_\ell h_\ell^{\tau p_t-d}\big)^{1/(1+p_t)}
    ,
\end{equation*}
and determine $N_0$ by equilibrating the two contributions
$C N_0^{-1/p_t} E = \cO(h_L^\tau)$, 
which yields
\begin{equation*}
    N_0
    = \cO\big((h_L^\tau(EC)^{-1})^{-p_t}\big) = \cO(2^{(L-\ell_0)tp_t}(EC)^{p_t})\;.
\end{equation*}

Since we consider interlaced polynomial lattice rules with $N_\ell=b^{m_\ell}$ points,
we choose $m_\ell$ for $b=2$ as
\begin{equation}\label{eq:ml}
    m_\ell = 
    \begin{cases}
        \lceil p_t(\tau(L+\ell_0) + \log_2 E ) - \frac{p_t}{1+p_t}(\ell_0(\tau+d) + \log_2 s_0) \rceil & \ell=0\\
        \big\lceil m_0 - \frac{p_t}{1+p_t} (\ell(\tau+d)+\log_2 (s_\ell/s_0) ) \big\rceil & \ell=1,\ldots,L
    \end{cases}
    .
\end{equation}
\subsubsection{Multilevel Monte Carlo (MLMC) sampling}
\label{sec:MC}
An alternative to QMC evaluation of the ratio and splitting estimators
is to approximate the prior expectations $Z$ and $Z'$ by MC sampling.
In MC, the convergence rate is $O(N^{-1/2})$ (in mean square, however)
in terms of the number $N$ of MC samples, 
independent of the summability exponents $p_0, p_t$.
The MLMC variant of MC appears as a particular case of the above
algorithms; a general optimization of the level-dependent sample numbers 
$N_\ell$ yields the same truncation dimensions $s_\ell$ as in \eqref{eq:sl} 
above and the following choice of number of samples.
Note, however, that in MC sampling we do not simply set $p_0=p_t=2$,
since the truncation dimensions $s_\ell$ still depend on the value of 
the summability exponent $0<p_0<1$. 
We obtain with $E_{MC} := \sum_{\ell=0}^L (s_\ell h_\ell^{2\tau-d})^{1/3}$ 
and $C_{MC}=(h_0^{\tau+d} s_0^{-1})^{1/3}$
the sample numbers 
\begin{equation}\label{eq:NellMC}
    N_0 = 2^{2Lt} \lceil (C_{MC}E_{MC})^2 \rceil \;,
\quad 
    N_\ell = N_0 
\lceil
\left(h_\ell^{\tau+d} h_0^{-\tau-d} s_\ell^{-1} s_0\right)^{2/3}
     \rceil\;,\quad 
    \ell = 1,\ldots,L 
\;.
\end{equation}
%
\section{Model forward problems}
\label{sec:NumExp}
We present examples of forward problems and discretizations
which verify the abstract hypotheses of the foregoing error analysis.
In Section \ref{sec:VerHyp} we 
verify that properties ${\bf HCP}_t$, ${\bf PGStabC}$ and
${\bf ANPGC}_{t'}$ of Section \ref{sec:HolPG} hold for
rather general classes of linear, affine-parametric operator equations. 
Next, in Section \ref{sec:TestPrb}, we consider a concrete
linear elliptic problem in two space dimensions with a particular,
explicit selection of the uncertainty parametrization \eqref{eq:uviapsi}.
We remark that analogous hypotheses are valid for a host
of more general problems with high-dimensional uncertainty parametrization;
we mention only parametric systems of nonlinear ODEs \cite{HaSc11}, 
and PDE problems with domain uncertainty.
\subsection{Affine-parametric linear operator equations}
\label{sec:VerHyp}
We verify the abstract hypotheses of the QMC-PG 
error analysis in Section \ref{sec:HolPG} for the model problem \eqref{eq:PDE1} below; 
in doing so, rather than considering the particular problem \eqref{eq:PDE1},
we consider the more general linear, affine parametric operator equations as
considered in \cite{ScMCQMC12}.

For a regularity parameter $0 \leq t \leq \bar{t} \leq \infty$, 
let $\{ A_j \}_{j\geq 0} \subset \cL(\bcX_t,\bcY'_t)$ be a sequence
of bounded, linear operators which satisfy the following hypotheses.

\noindent
${\bf H1}_t$ 
$A_0$ is boundedly invertible, i.e.~$A_0^{-1} \in \cL(\bcY'_t, \bcX_t)$.

\noindent
${\bf H2}_t$ 
For $r = 0, 1,\ldots, t$, the sequences $\bsb_r = (b_{j,r})_{j\geq 1}$ of norms
$b_{j,r} := \max\{ \| A_0^{-1}A_j \|_{\cL(\bcX_r)},\| (A_0^*)^{-1}A_j^* \|_{\cL(\bcY_r)}\}$ 
satisfy $\bsb_r \in \ell^{p_r}(\IN)$
for summability exponents $0 < p_0 \leq p_1 \leq ... \leq p_t < 1$.

\noindent
${\bf H3}_t$ 
There holds $\| \bsb_t \|_{\ell^{p_t}(\IN)} < 1$. 

Under {\bf H1}  - {\bf H3}, we consider for parameter sequences 
$\bsy = (y_j)_{j\geq 1}\in U = [-1,1]^\IN$, and for
$\{ A_j \}_{j\geq 0} \subset \cL(\bcX,\bcY')$ 
the linear, affine-parametric operators $A(\bsy;q)$ given by 
\be\label{eq:Aaffparm} 
A(\bsy;q) = A(\bsy)q = A_0 q + \sum_{j\geq 1} y_j A_jq \;.
\ee
Evidently, then, the differential $D_q\cR$ of the
residual map $\cR(\bsy;q) := A(\bsy)q - f$ 
equals $A(\bsy)$, and the (uniform w.r.t. $\bsy\in U$)
bounded invertibility of $A(\bsy)$ implies 
a uniform (w.r.t. $\bsy$) inf-sup condition for $A(\bsy)$.
Hypotheses ${\bf H1}_t$ - ${\bf H3}_t$ imply, with a Neumann series argument, 
properties ${\rm {\bf HCP}_t}$ and ${\rm {\bf ANPGC}}_t$ in Section \ref{sec:HolPG} 
for the affine-parametric family \eqref{eq:Aaffparm} of linear operators.
Specifically, 
for every $f\in \bcY'_t$ and for every $\bsy\in U$ holds in $\bcX_t$ 
(with $B_j := A_0^{-1}A_j \in \cL(\bcX_t)$ for $j\geq 1$)
\be\label{eq:AaffInv}
q(\bsy) = (A(\bsy))^{-1}f = (I + \sum_{j\geq 1} y_j B_j)^{-1}A_0^{-1}f
\;.
\ee
Denoting by $\bcX_t$ and $\bcY_t$ ``complexifications'' of the function
spaces, the affine-parametric operator family $\{ A(\bsy) :\bsy \in U \}$
in \eqref{eq:Aaffparm} extends to the polydisc $\calD = \{ \bsz : |z_j|\leq 1 \}$,
and the Neumann series inversion formula \eqref{eq:AaffInv} 
remains valid for $\bsz = (z_j)_{j\geq 1} \in \calD$.
The extended family $A(\bsz)$ being affine-parametric is holomorphic
with respect to each variable $z_j$ and, by ${\bf H3}_t$, 
is also boundedly invertible in $\cL(\bcX_t,\bcY_t')$ 
for all $\bsz \in \calD_{\bsrho_t}$,
provided the poly-radius $\bsrho_t = (\rho_{t,j})_{j\geq 1}$ 
is $(\bsb_t,\eps)$-admissible, i.e.~
provided that \eqref{eq:rho_b} holds for some $\eps > 0$.
Since the mapping $A\mapsto A^{-1}$ is analytic at any $A\in \cL_{iso}(\bcX,\bcY')$,
\eqref{eq:AaffInv} defines a parametric solution family $\{q(\bsz): \bsz \in \calD_{\bsrho_t} \}$, 
which, for $\bsz \in \calD_{\bsrho_t}$ and some $\eps > 0$, is $(\bsb_t,p_t,\eps)$-holomorphic
with holomorphy domains $\cO_{\rho_j}$ in Definition \ref{def:peanalytic} being the 
polydiscs $\calD_{\rho_j}$. For every $(\bsb_t,\eps)$-admissible $\bsrho_t$
there holds the \emph{holomorphic continuability of the parametric solutions}
${\bf HCP}_{t}$ for some $t>0$, so that
\be\label{eq:CplxReg}
\sup_{\bsz \in \calD_{\bsrho_t}} \| q(\bsz) \|_{\bcX_t} 
\leq C(\bsb_t,\eps) \| f \|_{\bcY'_t} 
\;.
\ee
As for every $\kappa > 1$, 
$\calD_\kappa \supset \cT_\kappa$, Proposition \ref{prop:main1}
is applicable to $g(\bsy) := \phi(q(\bsy))$ for any QoI $\phi(\cdot)\in \cL(\bcX_{t'},\cZ)$.

The Neumann series argument used to verify \eqref{eq:CplxReg} 
also implies condition {\bf PGStabC} 
(Stability and Quasioptimality of PG for the complex parametric extension). 
To see this, we assume at hand a nested sequence 
$\{ (\bcX^\ell,\bcY^\ell) \}_{\ell \geq 0}$
of (dense in $\bcX\times \bcY$) finite-dimensional PG trial- and test function spaces,
which satisfy the \emph{nominal inf-sup condition}: 
there exists $\mu_0 > 0$
such that
\be\label{eq:NomInfSup}
\inf_{0\ne w \in \bcX^\ell} \sup_{0\ne v \in \bcY^\ell}
\frac{|  _{\bcY'}\langle A_0 w , v \rangle_{\bcY} |} {\| w \|_{\bcX} \| v \|_{\bcY}} 
\geq \mu_0
\;,
\quad 
\inf_{0\ne v \in \bcY^\ell} \sup_{0\ne w \in \bcX^\ell}
\frac{| _{\bcY'}\langle A_0 w , v \rangle_{\bcY}| } {\| w \|_{\bcX} \| v \|_{\bcY}} \geq \mu_0
\;.
\ee
With \eqref{eq:NomInfSup}, assumption ${\bf H3}_0$ implies 
{\bf PGStabC}, i.e.
Stability and Quasioptimality of the PG discretizations for the complex parametric
extension to $\calD_{\bsrho_0}$ for any 
$(\bsb_0,\eps)$-admissible polyradius $\bsrho_0$:
for such $\bsz \in \calD_{\bsrho_0}$ holds
$$
\begin{array}{rcl}
|_{\bcY'}\langle A(\bsz) w , v \rangle_{\bcY}| 
& = & \displaystyle
\left|
_{\bcY'}\left \langle \left( A_0 + \sum_{j\geq 1} z_j A_j \right) w , v \right\rangle_{\bcY}
\right|
\geq 
\left|
_{\bcY'}\left \langle A_0 w , v \right\rangle_{\bcY}
\right|
- 
\left|
_{\bcY'}\left \langle \left( \sum_{j\geq 1} z_j A_j \right) w , v \right\rangle_{\bcY}
\right|,
\end{array}
$$
which implies with \eqref{eq:NomInfSup} and with $|z_j|\leq \rho_j$ that
\be\label{eq:ParInfSup1}
\forall \bsz \in \calD_\bsrho:\;\;
\inf_{0\ne v \in \bcY^\ell} \sup_{0\ne w \in \bcX^\ell}
\frac{| _{\bcY'}\langle A(\bsz) w , v \rangle_{\bcY}|} {\| w \|_{\bcX} \| v \|_{\bcY}} 
\geq \mu_0 (1 - \sum_{j\geq 1} \mu_0^{-1} \rho_j \| A_j \|_{\cL(\bcX,\bcY')})\;. 
\ee
Assume that $\bsrho$ is $(\bsbeta,\eps)$-admissible 
for the sequence $\bsbeta = (\beta_j)_{j\geq 1}$
given by $\beta_j := \mu_0^{-1} \| A_j \|_{\cL(\bcX,\bcY')}$ for
$j\geq 1$, and for sufficiently small $\eps > 0$.
Hypothesis ${\bf H1}_0$, i.e.~that $A_0 \in \cL(\cX,\cY')$ is an isomorphism,
implies 
$b_j = \| A_0^{-1} A_j \|_{\cL(\cX)} \simeq \|A_j\|_{\cL(\cX,\cY')} \simeq \beta_j$.
By {\bf H2} and {\bf H3}, 
$\bsb\in \ell^{p_0}(\IN)$ implies with {\bf H1} that $\bsbeta\in \ell^{p_0}(\IN)$. 
The $(\bsbeta,\eps)$-admissibility of $\bsrho$ and ${\bf H3}_t$ with $t=0$ imply
$$
1 - \sum_{j\geq 1} \mu_0^{-1} \rho_j \| A_j \|_{\cL(\bcX,\bcY')}
= 
1- \sum_{j\geq 1} \beta_j - \sum_{j\geq 1} (\rho_j-1)\beta_j
\geq 
1 - \eps - \sum_{j\geq 1} \beta_j 
> 0, 
$$
which implies in \eqref{eq:ParInfSup1} 
the uniform (w.r.t. $\bsz \in \calD_\bsrho$)
parametric discrete inf-sup conditions
\be\label{eq:ParInfSup}
\inf_{0\ne v \in \bcY^\ell} \sup_{0\ne w \in \bcX^\ell}
\frac{| _{\bcY'}\langle A(\bsz) w , v \rangle_{\bcY}| } {\| w \|_{\bcX} \| v \|_{\bcY}}
\geq \mu,
\;\;
\inf_{0\ne w \in \bcX^\ell} \sup_{0\ne v \in \bcY^\ell}
\frac{| _{\bcY'}\langle A(\bsz) w , v \rangle_{\bcY}| } {\| w \|_{\bcX} \| v \|_{\bcY}}
\geq \mu,
\ee
with $\mu = \mu_0 (1 - \eps - \sum_{j\geq 1} \beta_j) > 0$ independent of $\ell$.
The uniform parametric inf-sup conditions \eqref{eq:ParInfSup} imply 
uniform w.r.t. $\bsz\in\calD_\bsrho$ stability of the PG discretization
and the existence, uniqueness and the quasioptimality 
of the complex-parametric PG solution $q^{(\ell)}(\bsz) \in \bcX^\ell$,
for every $\bsz \in \calD_\bsrho$ with $(\bsbeta,\eps)$-admissible $\bsrho$.
It also implies the (uniform w.r.t. the discretization level $\ell$) 
$(\bsbeta,p_0,\eps)$-holomorphy of the PG solutions $q^{(\ell)}(\bsz)$.
The quasioptimality and the uniform parametric regularity \eqref{eq:CplxReg} imply,
with the approximation property \eqref{eq:apprprop}, the 
(uniform w.r.t. $\bsz$) asymptotic error bounds
\be\label{eq:ErrBdC}
\sup_{\bsz \in \calD_\bsrho} \| q(\bsz) - q^{(\ell)}(\bsz) \|_{\bcX} 
\leq C h_\ell^t \| f \|_{\bcY'_t} \;,
\ee
where the constant $C>0$ is independent of $\ell$, $f$ 
and of $\bsz\in \calD_\bsrho$, 
but depends on $\bsbeta$ and $\eps$ in the $(\bsbeta,\eps)$-admissibility of 
the poly-radius $\bsrho$. Choosing $\bsz = (z_1, ..., z_{s_\ell},0,...)$,
\eqref{eq:ErrBdC} implies \eqref{eq:PGerr}.
For the linear, affine-parametric operators 
\eqref{eq:Aaffparm}, property ${\rm {\bf ANPGC}}_{t'}$ 
for some $0\leq t' \leq \bar{t}$ was shown in \cite{DKGS14},
for the affine-parametric, linear operator family \eqref{eq:Aaffparm}.
\subsection{Affine-parametric linear test problem}
\label{sec:TestPrb}
We test the proposed ML algorithms for a
model parametric, linear diffusion problem which was
already considered in \cite{DKGS14}.

For a parameter $\bsy\in U=[-\frac12,\frac12]^\bbN$, 
in the bounded spatial domain $D\subset \mathbb{R}^d$,
we consider the second order, linear and affine-parametric 
elliptic PDE
\begin{align} \label{eq:PDE1}
 A(u;q) = &-\nabla \cdot \left(u(\bsy)\nabla q(\bsy)\right) \,=\, f\;, \quad
  q(\bsy)|_{\partial D} = 0\;,\quad
  u(\bsy) \,=\, u_0(\cdot) + \sum_{j\ge 1} y_j\,\psi_j(\cdot)\;.
\end{align}
This is a particular (linear) forward problem \eqref{eq:NonOpEqn}.
The usual (symmetric) primal variational formulation of \eqref{eq:PDE1} 
is based on $\cX=\cY=H^1_0(D)$.
In the particular case that $D=(0,1)^d$,
we parametrize the uncertain diffusion coefficient $u$ with the 
explicit, separable \KL basis
\begin{equation}\label{eq:sinuseig}
 \lambda_{\bsk} \,=\, \pi^2 (k_1^2+\cdots+k_d^2),\;\;
 \tilde{\psi}_{\bsk}(\bsx) \,=\, \prod_{i=1}^d \sin(\pi k_i x_i)
\;.
\end{equation}
Enumerating $\{\lambda_{\bsk}\}_{\bsk\in \bbN^d}$ in non-decreasing order
$\{\lambda_j\}_{j\ge 1}$, there holds
\begin{equation}\label{eq:Weyl}
\lambda_j \sim j^{2/d} \quad\mbox{as}\quad j\to\infty \;.
\end{equation}
We consider in the ensuing numerical experiments 
the case $d=2$, $D=(0,1)^2$, and
\begin{align}\label{eq:coeff}
 u(\bsy)(\bsx)
 &\,=\, u_0(\bsx) + \sum_{k_1, k_2 = 1}^\infty y_{k_1, k_2}\, \frac{1}{(k_1^2+k_2^2)^2} 
\,
 \sin( k_1 \pi x_1)\, \sin(k_2 \pi x_2) \nonumber \\
 &\,=\, u_0(\bsx) + \sum_{j=1}^\infty y_j\, \mu_j\, \sin (k_{1,j} \,\pi x_1)\, \sin(k_{2,j} \,\pi x_2) \;.
\end{align}
We enumerate the sequence of pairs 
$ \left( (k_{1,j}, k_{2,j}) \right)_{j \in \mathbb{N} }$ 
such that $k_{1,j}^2+k_{2,j}^2 \le k_{1,j+1}^2 + k_{2,j+1}^2$
for all $j \in \mathbb{N}$ 
(with ties due to equality broken in an arbitrary manner). 
This ordering yields 
$\mu_j = (k_{1,j}^2 + k_{2,j}^2)^{-2} \asymp \lambda_j^{-2} \sim j^{-2}$ (cf. \eqref{eq:Weyl}). 
We take $u_0(\bsx) \equiv 1/2$. 
In \eqref{eq:PDE1}, we use the forcing term $f(\bsx) = 100x_1$, and
consider the quantity of interest to be the integral of
the parametric solution $u(\bsy)$ over the subdomain 
$\bar{D}=(0.5,1)^2\subset D$, i.e.,
$G(q(\bsy)) = \int_{\bar{D}} q(\bsy)(\bsx) \,\rd\bsx$.
This affine-parametric forward problem fits into the abstract MLHoQMC framework 
considered in \cite{DKGS14} with symmetric, 
affine-parametric bilinear form $\fa(\bsy;\cdot,\cdot)$,
and with $\bcX = \bcY = H^1_0(D)$, and with
\begin{equation}\label{eq:ParModPrb}
  d = 2, \quad 
  \theta = 2, \quad 
  t=t'=1, \quad
  \tau = t + t' = 2, \quad 
  \mbox{and any} \quad 
  \frac{1}{2} < p_0 \leq 1 \;.
\end{equation}
As explained in \cite{DKGS14} this implies 
the summability exponent $p_1 = p_0/(1-p_0/2) > 2/3$.
The regularity spaces $\cX_t$ in the convergence rate estimate
\eqref{eq:convrate} are $\bcX_1 = (H^1_0\cap H^2)(D)$ and $\bcY'_1 = L^2(D)$.

For Bayesian inversion,
we consider as observation functional the integral over $\hat{D}=(0,0.5)^2$.
This scalar value is perturbed by a realization of a normally distributed
random variable $\eta\sim\mathcal{N}(0,\Gamma)$ as in \eqref{eq:DatDelta}
to generate a measurement, which is fixed for each value of $\Gamma$ before
applying the various SL and ML HoQMC methods.
For prior mean approximations,
we compare the multilevel QMC method to the single-level QMC approach
and to the multilevel Monte Carlo method.
For approximations to the posterior expectation,
we compare the performance of the analyzed
multilevel QMC estimators \eqref{eq:QLratio}, \eqref{eq:QLsplit}
with both the
single-level QMC ratio estimator of \cite{DGLGCSBIPSL}
as well as the two multilevel estimators \eqref{eq:QLratio}, \eqref{eq:QLsplit}
combined with standard Monte Carlo sampling.
In all considered algorithms, we solve \eqref{eq:PDE1} by 
a Galerkin approximation based on continuous, piecewise
linear finite elements on a family of uniform quadrilateral meshes 
with mesh width
$h_\ell= h_0 2^{-\ell}$ for $\ell = 0,1,2\ldots$, 
and we use interlaced polynomial lattice rules with $N=2^m$ points, 
$m\in\bbN$, constructed by
the fast CBC algorithm for SPOD weights from \cite{DKGNS13}.

In the single-level HoQMC ratio estimator,
the meshwidth is $h = h_L = h_02^{-L}$.
The PG discretization error for regular functionals is, asymptotically,
as $h\to 0$, $\calO(h^2)$. 
We balance this $\calO(h^2)$ discretization error with the 
dimension truncation error of $\calO(s^{-2})$ and the 
HoQMC quadrature error of $\calO(N^{-2})$.
These asymptotic error bounds yield the choices 
$s=h^{-1}=2^{L+1}$ and $N=h^{-1}$, 
so that $m=\log_2(h^{-1}) \simeq L+1$. 
Ignoring logarithmic factors,
this yields a combined error of $\calO(h^2) = \calO(\varepsilon)$ 
and overall cost of $\calO(Nh^{-2}s) = \calO(h^{-4})= \cO( \varepsilon^{-2})$,
with the constants implied in $\calO(\cdot)$ being independent of 
$\{s_\ell\}_{\ell \geq 0}$.

The HoQMC rules will be based on SPOD weights from \cite{KSS12}.
A major finding of the single-level theory in \cite{DGLGCSBIPSL} and 
of the multilevel error analysis in the present paper
is that \emph{HoQMC rules which are efficient for forward UQ
will perform equally well for the corresponding Bayesian Inverse UQ},
due to preservation of holomorphy domains.
We therefore use, in the affine parametric forward problem,
the HoQMC weights derived in \cite{DKGNS13}. 
They are given by \cite[Equation (3.32) with (3.17)]{DKGNS13} and 
\[
\mbox{
interlacing factor $\alpha = \lfloor 1/p_0 \rfloor + 1  = 2$, 
and}\;
\beta_j = \beta_{0,j} = \lambda_j = \frac{1}{(k_{1,j}^2 + k_{2,j}^2)^2} 
\;.
\]
The generating vectors were computed by
the fast CBC construction from \cite{DKGNS13} 
with Walsh constant $C=0.1$.
(computations with 
$C=1$ yielded different generating vectors,
and led to slight artefacts on high levels $L\ge7$ in this example). 
For the presently used base $b=2$, the choice $C=1.0$ holds \cite{Yoshiki15}.

In the multilevel algorithm $Q^\ast_L$ in 
\eqref{eq:MLratioZ'}, \eqref{eq:MLratioZ},
for given maximal discretization level $L$,
we take regular bisection refinement of the quadrilateral
mesh in $D$, resulting in a sequence of regular, quadrilateral meshes
of $D$ with meshwidths $h_\ell = h_0 2^{-\ell}$ for $\ell=0,1,\ldots,L$.
We 
select the truncation dimension as 
$s_\ell = \min(2^{2\ell},2^L)$ as in \eqref{eq:sl},
and $m_\ell$ as in \eqref{eq:ml}, 
where 
for this particular case $N_\ell$ is given in Table~\ref{tab:Nl}.
\begin{table}[h]
    \centering
    \begin{tabular}{lll}
        $\boldsymbol{L}$  &  \textbf{MLQMC}  &  \textbf{MLMC} \\ \toprule
         0   &   (1)  &   (1) \\
         1   &   (3,1)  &   (5,2) \\
         2   &   (5,3,1)  &   (10,7,4) \\
         3   &   (7,5,3,1)  &   (15,11,8,5) \\
         4   &   (9,7,5,3,1)  &   (19,15,12,9,7) \\
         5   &   (11,9,6,5,3,2)  &   (24,20,16,13,10,8) \\
         6   &   (13,11,8,6,5,3,2)  &   (28,24,20,17,14,11,9) \\
         7   &   (15,13,10,8,6,5,3,2)  &   -- \\
         8   &   (17,15,12,10,8,6,5,3,2)  &   -- \\
        \bottomrule
    \end{tabular}
    \caption{
        Logarithm of the number of samples per level $m_\ell = \log_2 N_\ell$
        for the multilevel methods used in the results below,
        i.e.~$N_\ell=2^{m_\ell}$ with $m_\ell$ given by \eqref{eq:ml} for QMC
        for both ratio and splitting estimators;
        for MLMC, we show $m^{MC}_\ell = \lfloor\log_2 N_\ell\rfloor$ with $N_\ell$ from \eqref{eq:NellMC}.
        Note that for MLMC we use $N_\ell$ directly,
        $m^{MC}_\ell$ is specified for ease of comparison.
    }
    \label{tab:Nl}
\end{table}
%
Using formally the limiting values $p_0=1/2$ and $p_1 = 2/3$,
the total error is $\calO(h_L^2) = \calO(\varepsilon)$ at cost of 
$\calO\left(\sum_{\ell=0}^L N_\ell h_\ell^{-2} s_\ell \right) 
 = \calO(\varepsilon^{-3/2})$, ignoring logarithmic factors. 
For $\ell=0$, we use the SPOD weights from the single-level 
case with $\beta_{0,j}$ from above. For $\ell>0$, 
the SPOD weights that enter the fast CBC construction are different from those
for the single-level algorithm; as indicated above, we use the choices
derived for this problem in \cite{DKGNS13} for the affine-parametric forward
problem also for the computation of integrals in the inverse problem.
We take base $b=2$, Walsh constant $C_{\alpha,b}=0.1$, 
and
\[
  \mbox{
  digit interlacing factor $\alpha = \lfloor 1/p_1 \rfloor + 1  =  2$, and}\;
  \beta_j = \beta_{1,j} = \lambda_j\,\pi\,\max(k_{1,j}, k_{2,j})
\;.
\]
\section{Implementation and numerical results}\label{sec:results}
We now present numerical experiments which validate
the choices of the algorithm steering parameters from Section \ref{sec:QMCNell}
for their implementation, for the fast, \emph{deterministic} solution
of Bayesian inverse problems for PDEs with uncertain random field inputs.
We present in particular experiments for 
the model linear, parametric elliptic forward problems from Section \ref{sec:NumExp}.
The problems are set in the domain $D=(0,1)^2$,
with homogeneous Dirichlet boundary conditions, and 
with parametric coefficients given by an $s$-term truncated
\KL expansion in $D$, with exactly known eigenfunctions.
The first problem has an affine-parametric coefficient,
the second problem is a nonlinear (holomorphic) transformation of this coefficient. 
The purpose of the ensuing numerical experiments
is to illustrate the preceding convergence analysis and to
show that the proposed MLHoQMC algorithms outperform
other methods, such as MLMC, in terms of error vs.~work.
For the implementation, we use the \texttt{gMLQMC} library from \cite{GaPASC16}.

We compute the forward solution up to maximal discretization level $L=8$,
yielding $s=512$ active dimensions.
No exact solution for this problem is available, so we 
verify convergence rates by testing accuracy with respect to
a numerically computed reference solution.
This reference solution was computed on level $L=9$ with
maximal truncation dimension $s=1024$
and the MLQMC method.
For the posterior expectation, the splitting estimator was used.
In the error vs.~work plot in the figures ahead,
we used the work measures
\begin{equation}\label{eq:WrkMeas}
W_{\rm SL} := h_L^{-2} s N\;, 
\quad \mbox{and} \quad 
W_{\rm ML} := \sum_{\ell=0}^L N_\ell h_\ell^{-2} s_\ell\;.
\end{equation}
The MLMC runs which are provided here for comparison purposes
were performed with identical discretizations of the forward problems,
and with the optimized MC sample numbers \eqref{eq:NellMC}; in order
to reduce the (inherent in MC sampling) scatter in the convergence
rate plots, in the ensuing graphs the MLMC convergence was obtained
by averaging $5$ MLMC runs. We emphasize that the presently considered
MLHoQMC are entirely deterministic: for the MLHoQMC algorithms,
the computation of each convergence plot required only one single run. 
%
\subsection{Affine parametric, linear elliptic test problem}
\label{sec:AffTstPbm}
In the results below, the work measures \eqref{eq:WrkMeas} were used for the multilevel methods.
The expected convergence rate is $-1/2$ for the single-level algorithm 
and $-2/3$ for the multilevel algorithm (for both forward and inverse approximations),
both of which are confirmed by the results.
In the MLMC runs, $R=5$ repetitions were used to average out sampling
noise in the estimated $L^2$ error.
The MLQMC (splitting estimator) result with finest discretization level $L=9$
was used as a reference in the error computation.
\begin{figure}[H]
    \centering
    \includegraphics[width=0.7\textwidth]{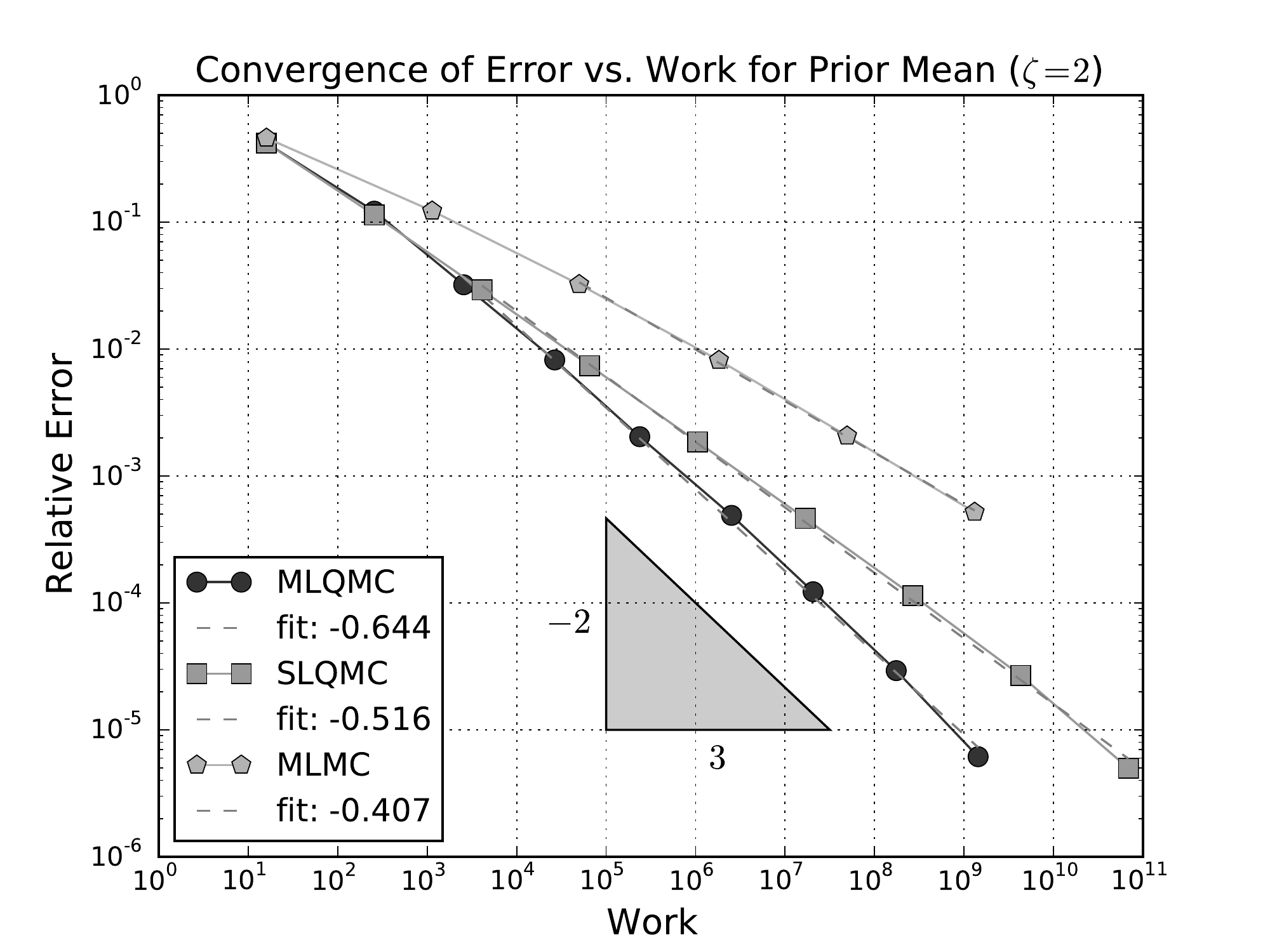}
    \caption{
        Convergence of forward UQ estimates for the affine-parametric
        2d diffusion equation model.
        The slope of the lines was estimated by a least-squares fit omitting the first three points.
    }
    \label{fig:affine_FW_G1}
\end{figure}
\begin{figure}[H]
    \centering
    \includegraphics[width=0.7\textwidth]{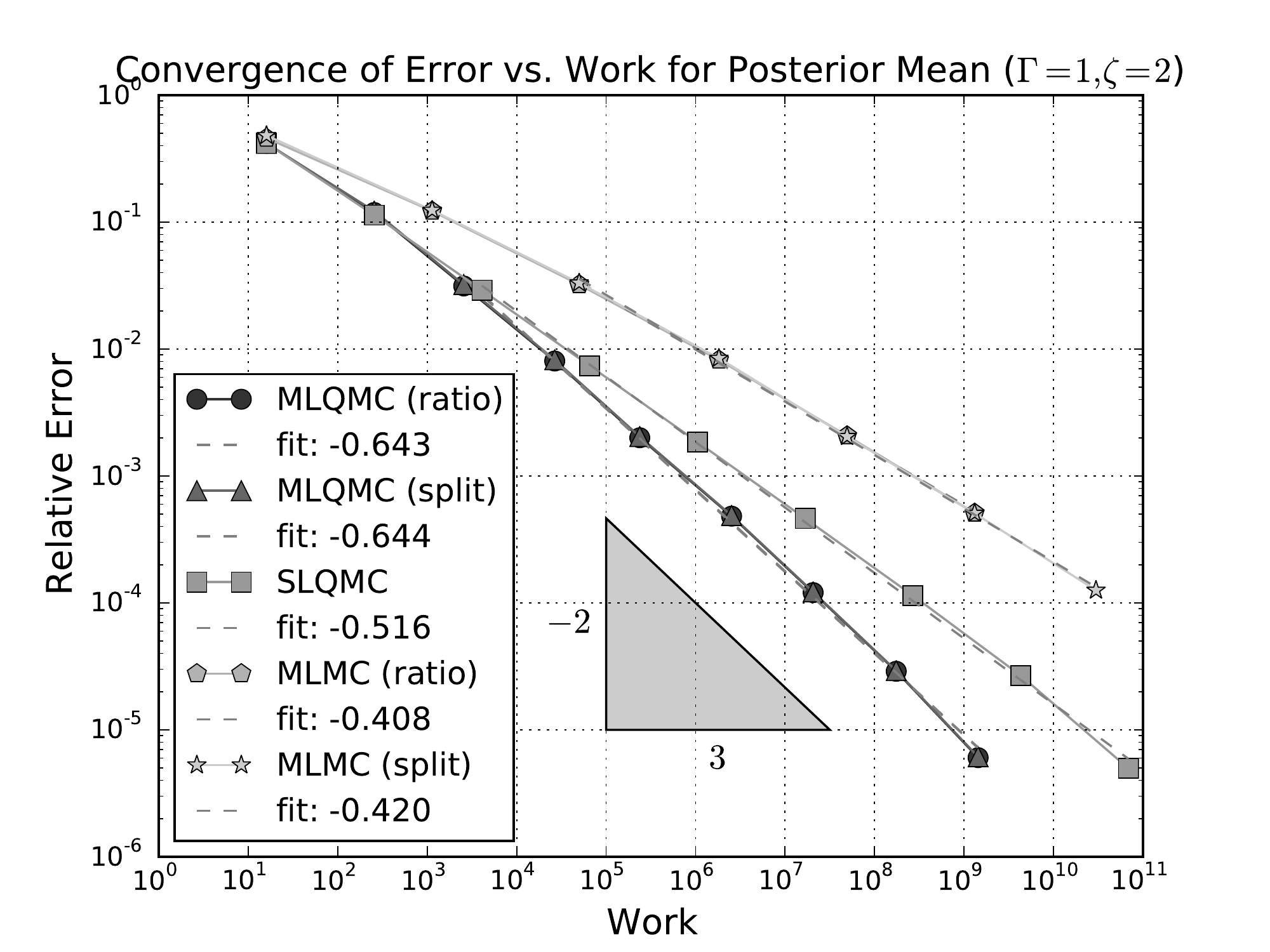}
    \caption{
        Convergence of estimates to the Bayesian inverse problem 
        for the affine-parametric 2d diffusion equation model, with covariance $\Gamma=1$.
        The slope of the lines was estimated by a least-squares fit omitting the first three points.
    }
    \label{fig:affine_BI_G1}
\end{figure}
\begin{figure}[H]
    \centering
    \includegraphics[width=0.7\textwidth]{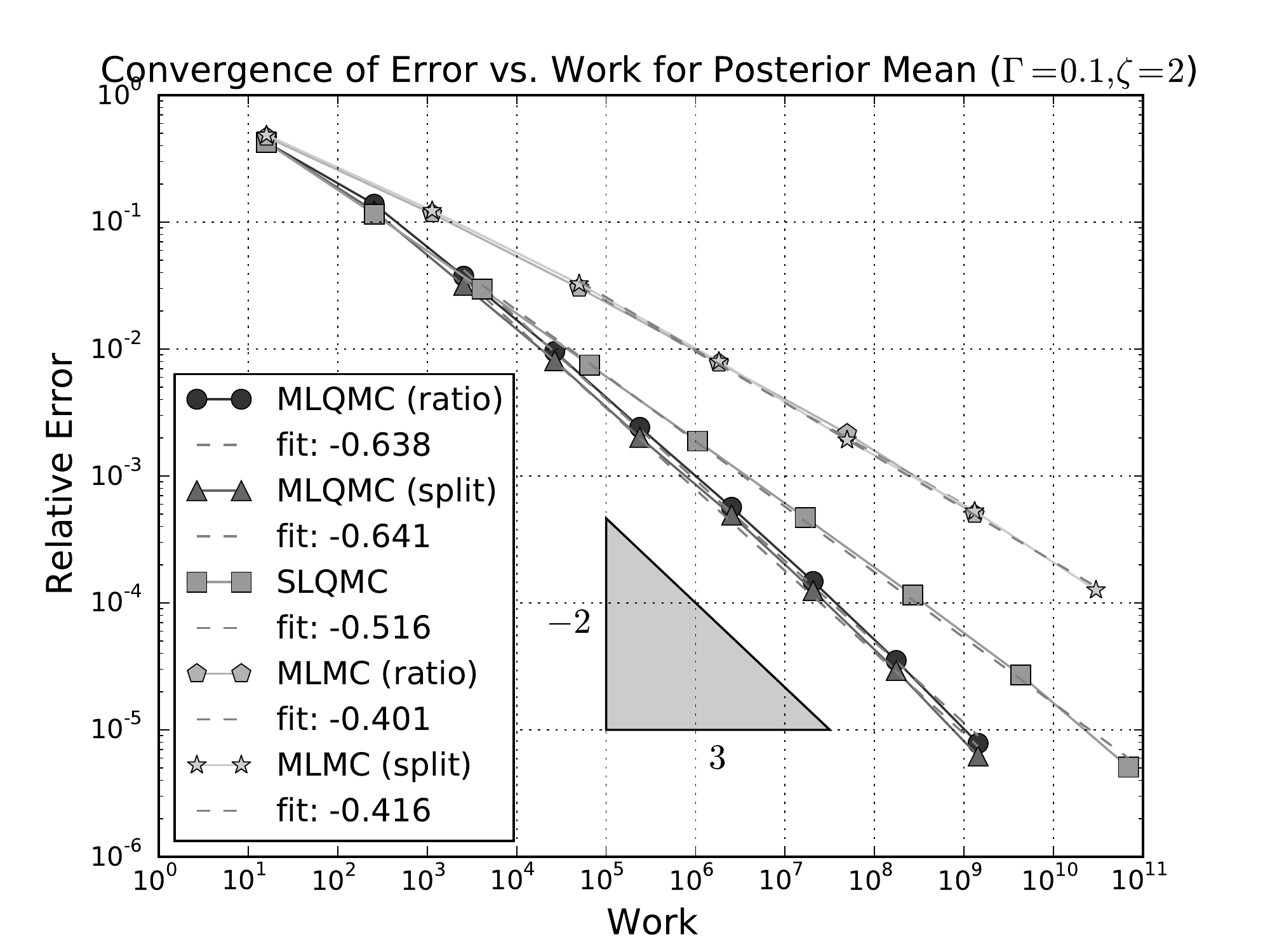}
    \caption{
        Convergence of estimates to the Bayesian inverse problem 
        for the affine-parametric 2d diffusion equation model, 
        with covariance $\Gamma=0.1$.
        The slope of the lines was estimated by a least-squares fit omitting the first three points.
    }
    \label{fig:affine_BI_G01}
\end{figure}
\begin{figure}[H]
    \centering
    \includegraphics[width=0.7\textwidth]{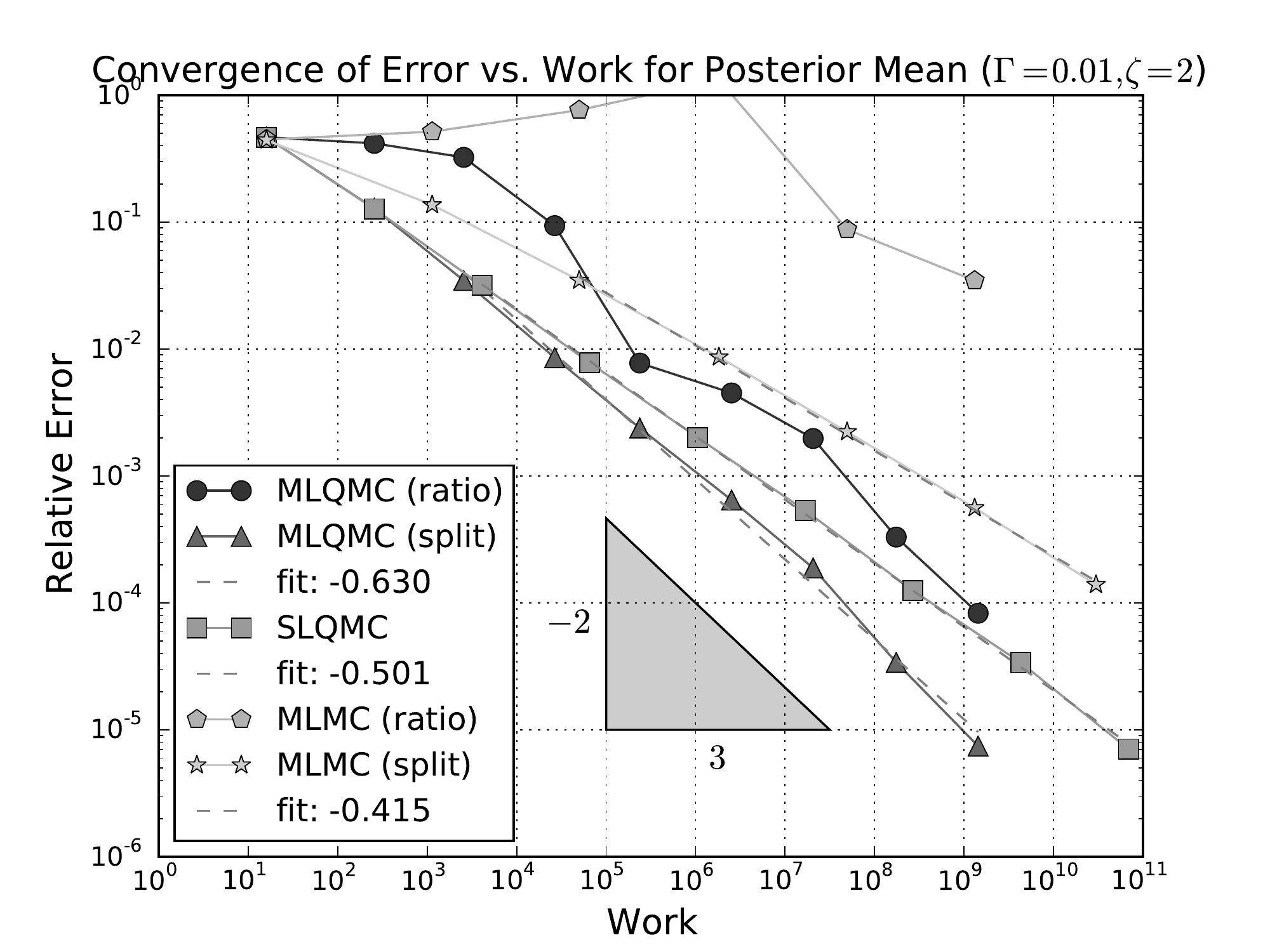}
    \caption{
        Convergence of estimates to the Bayesian inverse problem 
        for the affine-parametric 2d diffusion equation model, 
        with covariance $\Gamma=0.01$.
        The slope of the lines was estimated by a least-squares 
        fit omitting the first three points.
        In this case, the MLQMC splitting estimator was used as a reference.
        In the small noise case,
        the convergence of the ratio estimator is negatively impacted, 
        both for MLMC and for MLQMC, due to finite precision effects.
        The splitting estimator seems to be more resilient to rounding instabilities
        due to small $\Gamma$, as indicated in Remark \ref{rmk:Splt}.
    }
    \label{fig:affine_BI_G001}
\end{figure}

\subsection{Nonaffine-parametric, linear test problem}
We consider once more the linear operator equation 
\eqref{eq:PDE1}, however, now with diffusion coefficient modelled for
$y_j\in[-\frac12,\frac12]$ by the nonlinear expression
\begin{equation} \label{eq:Parmuexp}
    u(\cdot, \bsy) := \exp\left(\sum_{j\ge1} y_j \psi_j(\cdot) \right)\; ,
\end{equation}
i.e.~simply the exponential of the affine-parametric
coefficient model from \eqref{eq:PDE1} with $u_0(\cdot)=0$,
yielding for the new model the nominal value $u(\bszero) \equiv 1$.
When considering the exact same QMC parameters as for the above results ($\alpha,C,\beta_j$)
and maximal parameter dimension in the ML experiments being $s=1024$,
we observe the results in Figures~\ref{fig:MLFWnoaff_G1} to \ref{fig:MLBIPnoaff_G001}.

\begin{figure}[h]
    \centering
    \includegraphics[width=0.8\textwidth]{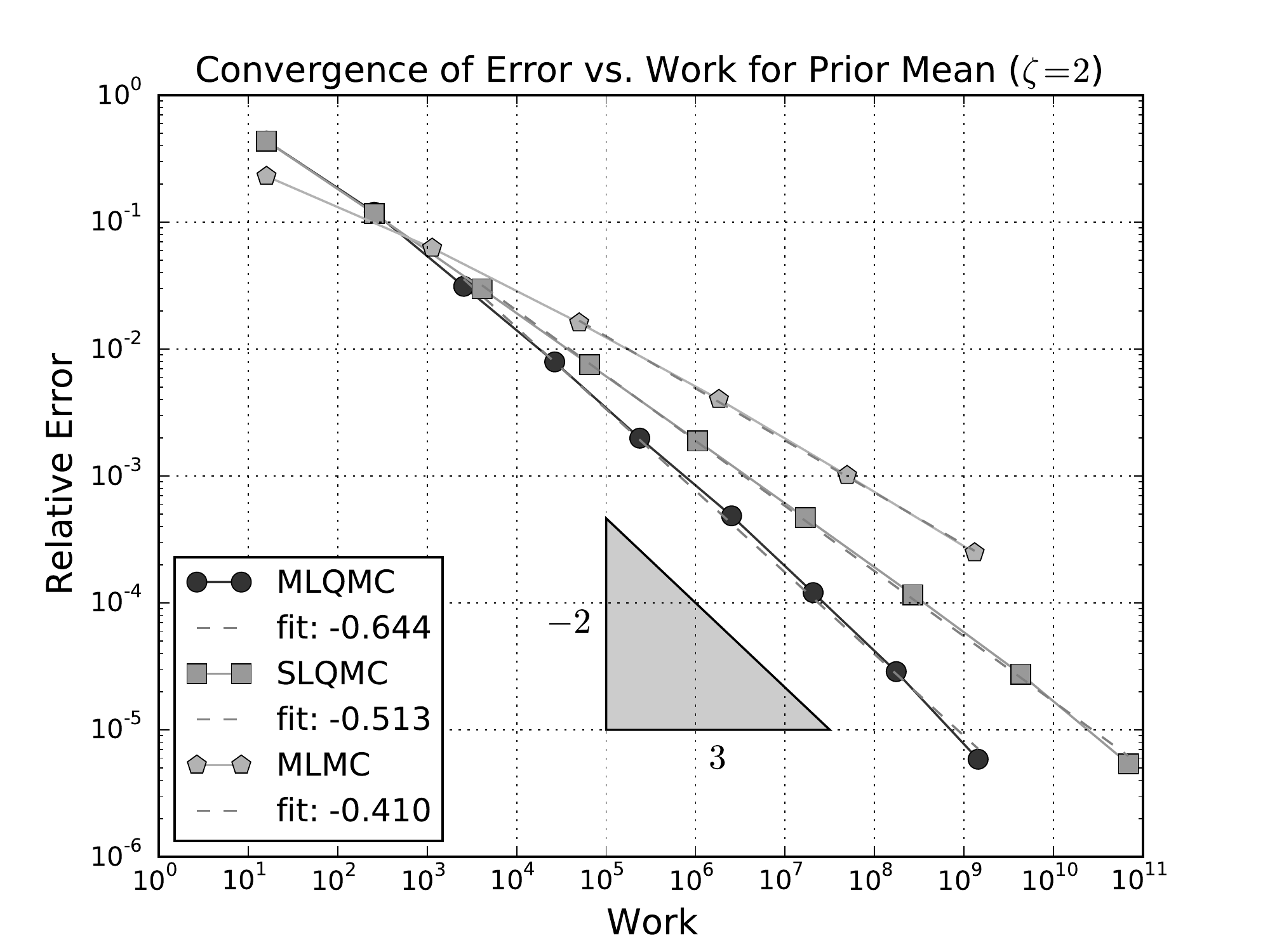}
    \caption{
        Convergence of forward UQ estimates for the nonaffine-parametric
        2d diffusion equation model with coefficient \eqref{eq:Parmuexp}.
        The slope of the lines was estimated by a least-squares fit omitting the first three points.
    }
    \label{fig:MLFWnoaff_G1}
\end{figure}
\begin{figure}[h]
    \centering
    \includegraphics[width=0.8\textwidth]{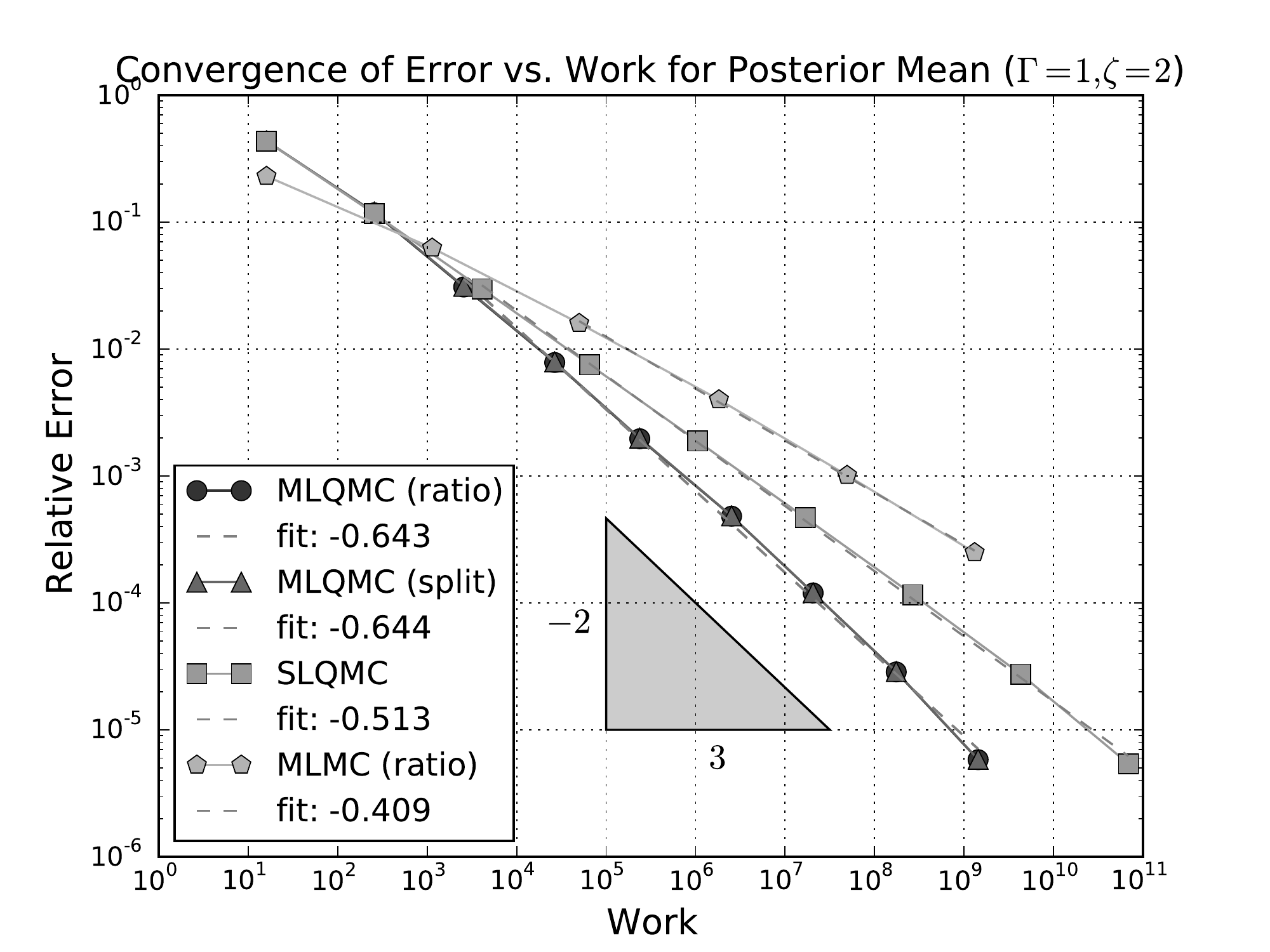}
    \caption{
        Convergence of estimates to the Bayesian inverse problem for the
        2d diffusion equation model with coefficient \eqref{eq:Parmuexp}, with $\Gamma=1$.
        The slope of the lines was estimated by a least-squares fit omitting the first three points.
    }
    \label{fig:MLBIPnoaff_G1}
\end{figure}
\begin{figure}[h]
    \centering
    \includegraphics[width=0.8\textwidth]{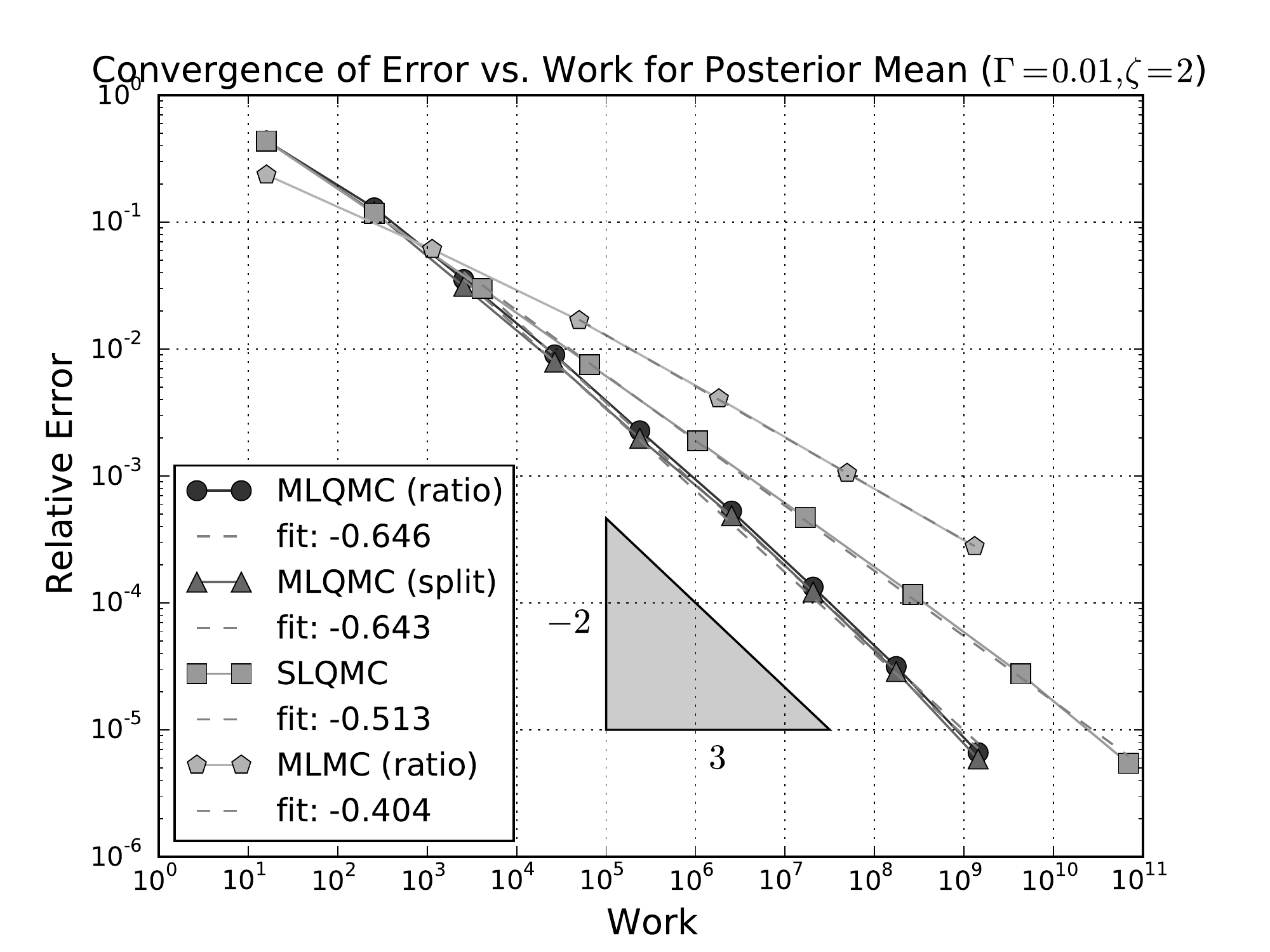}
    \caption{
        Convergence of estimates to the Bayesian inverse problem for the
        2d diffusion equation model with ``log-affine'' parametric 
        coefficient \eqref{eq:Parmuexp}, with $\Gamma=0.01$.
        The slope of the lines was estimated by a least-squares fit omitting the first three points.
    }
    \label{fig:MLBIPnoaff_G001}
\end{figure}
%
\section{Conclusions}
\label{sec:Concl}
We extended \cite{DGLGCSBIPSL} to a class of 
deterministic, multilevel Petrov-Galerkin, 
higher order Quasi-Monte Carlo integration algorithms 
for forward and Bayesian inverse computational uncertainty
quantification of possibly nonlinear, well-posed operator equations.
Novel, computable deterministic multilevel estimators have been proposed for 
``distributed'' uncertain input data in a separable Banach space $X$. 
Upon parametrizing the uncertain input data in terms of a countable
basis of $X$ (as, e.g., through a \KL expansion), 
and upon multilevel Petrov-Galerkin discretization of the forward problems, 
the forward and Bayesian inverse uncertainty quantification 
problem is reduced to numerical evaluation of high-dimensional,
parametric integrals of nonlinear functionals depending on the
likelihood function of the responses from the parametric forward problem 
and the observation data.
The numerical integration is conducted with the deterministic,
higher order QMC integrations from \cite{DLGCS14}.
The present results generalize, in particular, the 
HoQMC PG error analysis of \cite{DKGS14} to smooth, nonlinear operator equations 
with holomorphic-parametric dependence of their responses on the parameters.
They apply to broad classes of forward equations, 
with possibly indefinite or saddle point variational formulations,
and nonlinear, analytic dependence on the parameters.

In several numerical experiments for Bayesian inversion of linear, 
elliptic forward problems in two space dimensions, 
the presently proposed, multilevel higher order Quasi Monte-Carlo strategy
consistently outperformed the corresponding single-level
algorithms from \cite{DGLGCSBIPSL}, and corresponding MLMC
methods in both forward as well as
Bayesian inverse UQ on parametric inputs with 
proper sparsity: to reach one percent accuracy in the 
Bayesian estimate, the MLHoQMC strategy achieves a speedup of a factor $10$ 
of error versus total work as compared to the SL strategy and to the MLMC
approach.
Higher efficiency is expected on input data with higher smoothness
in the data space, i.e.~$u \in X_t$, implies higher sparsity; 
having said this, we admit that for problems whose parametric input data
does not afford sufficient sparsity, the presently proposed 
methods will not outperform MLMC and ML versions of first order QMC methods.

The presented numerical experiments also confirm the dimension-independence
of the QMC convergence rates, which are only limited by input sparsity
and by the digit interlacing order of the polynomial lattice rule.
They also indicate the expected deterioration of the algorithms' performance
for small observation noise covariance $\Gamma$;
in this respect, the splitting estimator was found to be less sensitive than the ratio estimator. 
The presently introduced algorithms allow us to handle uncertainties with 
several hundred to thousands of parameters, in two space dimensions,
with moderate computational effort. The parametric sparsity of the
countably parametric model problem considered in the numerical experiments
was moderate ($p_0 = 1/2$ in \eqref{eq:ParModPrb}); for 
classes of uncertain input data $u$ with higher sparsity, i.e.~smaller values of $p_0$, 
the gains of the presently proposed, HoQMC-based algorithms over MLMC
are predicted to be correspondingly higher, as a consequence of the present
theoretical results, and supported by extensive numerical experiments in \cite{GaCS14}.
\paragraph{Acknowledgments}
This work was supported by CPU time from the
Swiss National Supercomputing Centre (CSCS) under project IDs s522 and d41,
by the Swiss National Science Foundation (SNF) under Grant No. SNF149819, 
by the European Research Council (ERC) under AdG 247277, 
and by Australian Research Council's Discovery Project grants under project number DP150101770.
%
%

\end{document}